\documentclass[11pt]{article}
\usepackage[a4paper,margin=2cm]{geometry}
\usepackage[english]{babel}
\usepackage{latexsym,amsmath,enumerate,graphics,enumerate,amsthm,tikz,hyperref,float}
\usepackage[mathscr]{euscript}
\usepackage[affil-it]{authblk}
\usepackage{enumerate,amsthm,dsfont,pstricks}
\usepackage{latexsym,amsmath,amssymb,amscd,wrapfig,graphicx}
\usepackage{hyperref,cleveref}

\usepackage{soul, xcolor}

\newtheorem{theorem}{Theorem}[section]
\newtheorem{lemma}[theorem]{Lemma}
\newtheorem{proposition}[theorem]{Proposition}
\newtheorem{corollary}[theorem]{Corollary}

\theoremstyle{definition}

\newtheorem{definition}[theorem]{Definition}
\newtheorem{example}[theorem]{Example}

\theoremstyle{remark}
\newtheorem{remark}[theorem]{Remark}

\newtheorem*{theorem*}{{\bf Theorem}}

\newtheorem*{assumption*}{{\bf Assumption}}

\let\phi=\varphi
\def\N{\mathbb{N}}
\def\R{\mathbb{R}}
\def\Z{\mathbb{Z}}
\def\Q{\mathbb{Q}}
\def\F{\mathcal{F}}
\def\L{\mathcal{L}}

\def\I{\mathcal{I}}
\def\U{\mathcal{U}}

\def\eps{\varepsilon}

\def\ol{\overline}

\def\ua{\mathord{\uparrow}}
\def\sua{\mathord{\upharpoonright}}
\def\da{\mathord{\downarrow}}
\def\sda{\mathord{\downharpoonright}}

\DeclareMathOperator{\Lex}{Lex}
\DeclareMathOperator{\lex}{lex}
\DeclareMathOperator{\car}{car}

\DeclareMathOperator{\ran}{ran}
\DeclareMathOperator{\supp}{supp}

\newcommand{\comment}[1]{}

\newcommand{\wh}[1]{\widehat #1}

\numberwithin{equation}{section}

\def\sgn{\mathrm{sgn}}
\let\epsilon=\varepsilon

\makeatletter
\def\@maketitle{%
  \newpage
  \null
  \vskip 2em%
  \begin{center}%
  \let \footnote \thanks
    {\Large\bfseries \@title \par}%
    \vskip 1.5em%
    {\normalsize
      \lineskip .5em%
      \begin{tabular}[t]{c}%
        \@author
      \end{tabular}\par}%
    \vskip 1em%
    {\normalsize \@date}%
  \end{center}%
  \par
  \vskip 1.5em}
\makeatother

\begin{document}

\setstcolor{red}

\title{\sc \huge Artinian and Noetherian vector lattices}

\author{Marko Kandi\'{c}%
\thanks{Email: \texttt{marko.kandic@fmf.uni-lj.si}}}
\affil{Faculty of Mathematics and Physics, University of Ljubljana, Jadranska 19, SI-1000 Ljubljana, Slovenia}
\affil{Institute of Mathematics, Physics and Mechanics, Jadranska 19, SI-1000 Ljubljana, Slovenia}

\author{Mark Roelands%
\thanks{Email: \texttt{m.roelands@math.leidenuniv.nl}}}
\affil{Mathematical Institute, Leiden University, 2300 RA Leiden,
The Netherlands}

\author{Marten Wortel%
\thanks{Email: \texttt{marten.wortel@up.ac.za}}}
\affil{Department of Mathematics and Applied Mathematics, University of Pretoria, Private Bag X20 Hatfield, Pretoria 0028, South Africa}

\maketitle

\begin{abstract}

In this paper, we study Artinian and Noetherian properties in vector lattices and provide a concrete representation of these spaces. Furthermore, we describe for which Archimedean uniformly complete vector lattices every decreasing sequence of prime ideals is stationary (a property that we refer to as prime Artinian). We also completely characterize the prime ideals in vector lattices of continuous root functions and piecewise polynomials. This is a useful space for studying how having decreasing stationary sequences of prime ideals does not imply having increasing stationary sequences of prime ideals, and vice versa.
\end{abstract}



\section{Introduction} 
Order ideals in vector lattices have been, and still are, extensively studied. Properties of order ideals are closely related to the structure of the vector lattice that contains them. For instance, a vector lattice is finite-dimensional if and only if it has only finitely many order ideals; the existence of a strong unit implies the existence of maximal ideals, see \cite[Theorem 27.4]{Zaanen}; all order ideals are principal if and only if the Archimedean vector lattice is isomorphic to a space of functions with finite support on some set, see \cite[Theorem 5.2]{Kandic-Roelands}. The prime ideals in vector lattices and certain carefully chosen subsets thereof yield representations of the space, see \cite{Johnson-Kist} and \cite{Yosida}. Also, the interplay between order ideals and projection bands was studied in \cite{Huijsmans}. In vector lattices with additional algebraic structure, the relation between order ideals and algebraic ideals was studied in \cite{Huijsmans-dePagter1}, \cite{Huijsmans-dePagter2}, and \cite{dePagter}. Recently, a characterization of projection bands was given in \cite{EugeneBilokopytov}.

This paper studies the implications of stationary properties of order ideals. Similar to the theory of commutative rings, we call a vector lattice Artinian if every decreasing sequence of ideals is stationary, and similarly, if every increasing sequence of ideals is stationary, the vector lattice is called Noetherian. To provide a characterization of these vector lattices, we need the notion of local and local ideals, and the concept of induction and recursion on well-founded sets. Local vector lattices were studied extensively in \cite{Boulabiar19} and also in \cite{Schaefer}, but not explicitly. The main result describing Artinian vector lattices is \Cref{T:characterisation of Artinian vector lattices}, stating that every Artinian vector lattice is isomorphic to a lexicographical vector lattice $\mathrm{Lex}(X)$ for some reverse well-founded forest $X$ with finite width. Furthermore, \Cref{T:characterisation of Noetherian vector lattices} is the main result characterising Noetherian vector lattices. This result states that a vector lattice is Noetherian precisely when it is a sublattice of the lexicographical lattice $\R^X_{\mathrm{lex}}$ containing $\mathrm{Lex}(X)$ for some well-founded forest $X$ with finite width. A related result was proved by Hausner and Wendel in \cite{HausnerWendel} for totally ordered vector lattices. In this case there is a totally ordered set $X$ such that the vector lattice is a sublattice of $\R^X_{\mathrm{lex}}$ containing $\mathrm{Lex}(X)$. Another class of order ideals to study when considering stationary properties are the prime ideals, since all ideals containing a fixed prime ideal are all prime ideals and form a totally ordered set. Vector lattices for which decreasing sequences of prime ideals are stationary are called prime Artinian and the main result characterising Archimedean uniformly complete vector lattices that are prime Artinian is \Cref{T: prime Artinian unif complete}, stating that these are precisely the vector lattices isomorphic to the functions with finite support on some set. Similarly, a vector lattice is called prime Noetherian if every sequence of increasing prime ideals is stationary. Prime Noetherian vector lattices have been studied and characterized in \cite{Kandic-Roelands}. It turns out that an Archimedean uniformly complete vector lattice is prime Artinian if and only if it is prime Noetherian. The reason why these two notion imply each other crucially uses the fact that the vector lattice is uniformly complete. If we consider the vector lattice of continuous piecewise root functions and polynomials, then it can be shown that neither prime Artinian implies prime Noetherian nor prime Noetherian implies prime Artinian in general, the last of the two statements was already shown in \cite{Kandic-Roelands}. 

The structure of the paper is as follows. Section 2 is the preliminary part of this paper. Here we discuss some basic concepts in vector lattices, local ideals in vector lattices, followed by well-founded sets and how to apply induction and recursion on these sets. In Section 3 we mainly focus on representing Artinian and Noetherian vector lattices, and the first part of this section includes results characterizing when vector lattices are Artinian and Noetherian that do not use the above mentioned representations of the space. Section 4 considers decreasing sequences of prime ideals in Archimedean vector lattices, and we characterize prime Artinian Archimedean uniformly complete vector lattices. In Section 5 we study the class of piecewise root functions and polynomials for which we can explicitly compute the prime ideal structure and use this to show that it is generally not the case that prime Artinian vector lattices are prime Noetherian.

\section{Preliminaries}

\subsection{Basic concepts in vector lattices}
In this paper, $E$ will denote a vector lattice with cone $E_+$. We say that $E$ is \emph{Archimedean} if for any $x,y \in E_+$ such that $nx \le y$ for all $n \in \N$ it follows that $x = 0$. For any $x,y,z \in E_+$ satisfying $z \le x + y$, there are $z_1, z_2 \in E_+$ with $z_1 \le x$ and $z_2 \le y$ such that $z_1 + z_2 = z$. This is called the \emph{Riesz decomposition property}. A subspace $I \subseteq E$ is called an \emph{order ideal} or just \emph{ideal} if $|x| \le |y|$ for any $x \in E$ and $y \in I$ implies that $x \in I$. Given a nonempty set $S \subseteq E$, there is a smallest ideal that contains $S$. This ideal will be denoted by $I_S$  and is called the \emph{ideal generated by $S$}. Note that we may chose this set $S$ in $E_+$, as $\{|x|\colon x \in S\}$ generates the same ideal as $S$. If $S = \{x\}$, then $I_S$ is said to be \emph{principal} and we will write $I_x$ instead of $I_{\{x\}}$. If $x \in E_+$ is such that $I_x = E$, then we call $x$ a \emph{strong unit}. A proper ideal $M \subseteq E$ is called \emph{maximal} if for any ideal $J$ such that $M\subseteq J\subseteq E$ it follows that either $J = M$ or $J = E$. If the codimension of an ideal $I$ is one, then it is maximal, and conversely, all maximal ideals have codimension one by \cite[Corollary~p.66]{Schaefer}. A proper ideal $P \subseteq E$ is called \emph{prime} if for any $x,y \in E$ such that $x \wedge y \in P$ it follows that either $x \in P$ or $y \in P$. Furthermore, by \cite[Theorem~33.3]{Zaanen} every maximal ideal is prime. The following useful characterization of prime ideals is proved in \cite[Theorem~33.2]{Zaanen}.

\begin{theorem}
    Let $P$ be an ideal in a vector lattice $E$. Then the following statements are equivalent.
    \begin{itemize}
        \item[$(i)$] $P$ is prime.
        \item[$(ii)$] If $x \wedge y = 0$, then $x \in P$ or $y \in P$.
        \item[$(iii)$] The quotient vector lattice $E/P$ is totally ordered.
        \item[$(iv)$] For any ideals $I$ and $J$ such that $I \cap J \subseteq P$ it follows that $I \subseteq P$ or $J \subseteq P$.
    \end{itemize}
\end{theorem}

Given a prime ideal $P$, then every ideal containing $P$ must be prime by \cite[Theorem~33.3]{Zaanen} and furthermore the family of ideals containing $P$ is totally ordered. 

We say that a net $(x_\alpha)_\alpha$ converges \emph{relatively uniformly} to $x$ if there exists some $y \in E_+$, which regulates the convergence, so that for each $\eps > 0$ there is an index $\alpha_\eps$ with the property that $|x_\alpha - x| \le \eps y$ whenever $\alpha \ge \alpha_\eps$. If every relatively uniform Cauchy net converges relatively uniformly, with respect to the same regulator, then $E$ is called \emph{relatively uniformly complete}. If $E$ is Archimedean and given a vector $x \in E_+$, the principal ideal $I_x$ equipped with the norm 
\[
\|y\|_x := \inf \{ \lambda > 0 \colon -\lambda x \le y \le \lambda x\}
\]
is norm dense in its norm completion, which by the Kakutani representation theorem, \cite[Theorem~4.29]{AliprantisBurkinshaw}, is lattice isometric to the Banach lattice $C(K)$ for some compact Hausdorff space $K$. The strong unit $x$ in $I_x$ is then mapped to the constant one function. Hence, if $I_x$ is complete, then it is itself lattice isometric to a $C(K)$-space. Note that a vector lattice $E$ is relatively uniformly complete if and only if $(I_x,\|\cdot\|_x)$ is a Banach lattice for each $x \in E_+$. 

We refer the reader to the standard texts \cite{Zaanen}, \cite{Schaefer}, \cite{AliprantisBurkinshaw}, \cite{AbramovichAliprantis} for further reading or more details about the above mentioned concepts and terminology.  

\subsection{Local ideals}

In this section we discuss an important class of ideals that will be used to describe Noetherian and Artinian vector lattices.

\begin{definition}
Let $E$ be a vector lattice and $L \subseteq E$ an ideal. Then $L$ is defined to be a \emph{local} ideal if $L$ has a proper ideal $M$ containing all other proper ideals in $L$. The space $E$ is defined to be a local vector lattice if $E$ is a local ideal. 
\end{definition}

Local vector lattices were studied in detail in \cite{Boulabiar19} but were already implicitly present in \cite{Schaefer}. Indeed, \cite[Proposition~II.3.7]{Schaefer} shows that $E$ is local with maximal ideal $M$ if and only if $E \cong \R \circ M$, where $\R \circ M$ denotes the lexicographic union: the vector space $\R \times M$ ordered by the cone of elements $(\lambda, x)$ with $\lambda > 0$, or $\lambda = 0$ and $x \in M_+$. The element $(1,0)$ corresponds to a strong unit of $L$. Note that in the isomorphism $E \cong \R \circ M$ the maximal ideal $M$ is unique, but the copy of $\R$ is not. Indeed, any positive element in $E$ not in $M$ can be chosen to span this copy of $\R$, as can be easily seen in the example $E = \R^2_{\lex}$, where $M$ equals the $y$-axis and the copy of $\R$ can be any other line through the origin. In this paper we often choose isomorphisms $I \cong \R \circ M$ for all local ideals $I$ of a vector lattice, therefore the Axiom of Choice is unavoidable in our work.

\begin{lemma}\label{L:local_ideals_comparable_or_disjoint}
Let $E$ be a vector lattice and let $L_1$ and $L_2$ be local ideals in $E$. Then either $L_1$ and $L_2$ are comparable, or $L_1$ and $L_2$ are disjoint.
\end{lemma}

\begin{proof}
Suppose $L_1 \cong \R e_1 \circ M_1$ and $L_2 \cong \R e_2 \circ M_2$ (for some positive $e_1$ and $e_2$) are not comparable. Let $u = e_1 \wedge e_2 \geq 0$. Since $u \leq e_1$, $u \in L_1$ and so $u = \lambda e_1 + u_1$ with $u_1 \in M_1$. Since $u \geq 0$, $\lambda \geq 0$. Suppose $\lambda > 0$. Then $\lambda e_1 + u_1 = u \leq e_2$ and since $\lambda e_1 + u_1$ is an order unit of $L_1$, $L_1 \subseteq L_2$, contradicting the incomparability of $L_1$ and $L_2$. Hence $u = u_1 \in M_1$, and by a similar argument, $u \in M_2$. Suppose $u \not= 0$. Then $2u > u$, $2u \leq e_1$ since $2u \in M_1$, and $2u \leq e_2$ since $2u \in M_2$. This contradicts the fact that $u$ is the greatest lower bound of $e_1$ and $e_2$. Hence $e_1 \wedge e_2 = 0$ and so $L_1$ and $L_2$ are disjoint,  since $e_i$ is a strong unit in $L_i$.
\end{proof} 

The following property of local ideals will be useful later when representing Noetherian vector lattices. 

\begin{proposition}\label{P:local_ideal_in_summand}
Let $E \cong \bigoplus_{i=1}^n L_i$ be a vector lattice where $n \geq 1$ and $L_1, \ldots, L_n$ are local ideals. If $L$ is a local ideal of $E$, then there is a $k$ with $L \subseteq L_k$, and furthermore, $L = L_k$ or $L \subseteq M_k$.
\end{proposition}
 \begin{proof}
 If $L$ is disjoint from all $L_k$'s, then $L = \{0\}$ which is a contradiction. Therefore there is a $k$ such that $L$ is not disjoint from $L_k$, and so $L$ is comparable with $L_k$ by Lemma~\ref{L:local_ideals_comparable_or_disjoint}. If $L \subseteq L_k$, then if $L \not\subseteq M_k$, there is an $x \in L$ with $x = \lambda e_k +y$ with $\lambda \not= 0$ and $y \in M_k$. By passing to $|x|$ if necessary, we may assume that $\lambda > 0$. Hence $L$ contains an order unit of $L_k$ and so $L = L_k$. It remains to show that $L_k \subsetneq L$ is impossible.
 
 Suppose $L_k \subsetneq L$. This implies that $L$ is not disjoint with some $L_j$ for some $j \not= k$. By Lemma~\ref{L:local_ideals_comparable_or_disjoint}, either $L \subseteq L_j$ or $L_j \subseteq L$; the first is impossible since $L$ contains $L_k$. Hence $L_j \subseteq L$. Now $L \cap [(\bigoplus_{i \not= k} L_i) \oplus M_k]$ and $L \cap[(\bigoplus_{i \not= j} L_i) \oplus M_j]$ are two different codimension one ideals in $L$, so they are maximal. This contradicts the fact that $L$ is local.
 \end{proof}

\subsection{Well-founded posets}
This section discusses well-founded posets, which are an essential ingredient for characterizing Artinian and Noetherian vector latices.

Let $X$ be a poset. For $x \in X$, we define $\da x := \{y \in X \colon y \leq x\}$ and $\sda x := \{y \in X \colon y < x\}$, and for $A \subseteq X$ we define $\da A := \cup_{x \in A} \, \da x$. The notations $\ua x$, $\sua x$, and $\ua A$ are defined similarly, and $U \subseteq X$ is defined to be an \emph{upper set} if $\ua U = U$. For $A \subseteq X$, we denote by $\min(A) := \{x \in A \colon \sda x \subseteq A^c\}$ the set of minimal elements of $A$. A subset $A \subseteq X$ is an \emph{antichain} if no elements of $A$ are comparable.

We will next discuss well-founded induction and recursion. This topic is treated in full generality for classes in any introductory text on set theory (see for example \cite[Section~III.5]{Kunen}), but we only need the results for sets, so for the convenience of the reader we give an elementary explanation involving sets only.

\begin{definition}
Let $X$ be a poset (or a subset of a poset). Then $X$ is called \emph{well-founded} if every nonempty subset of $X$ has a minimal element.
\end{definition}

\subsubsection{Well-founded induction}

Transfinite induction readily generalizes to well-founded posets.

\begin{lemma}\label{L:well-founded equivalences}
Let $X$ be a poset. Then the following statements are equivalent.
\begin{enumerate}
    \item $X$ is well-founded.
    \item For every nonempty subset $A$ of $X$, every element of $A$ dominates a minimal element of $A$.
    \item Every decreasing sequence in $X$ is stationary.
    \item $X$ has no strictly decreasing sequences.
    \item The principle of well-founded induction is valid for $X$: if $A \subseteq X$ is such that for all $x \in X$, the property that all $y < x$ are in $A$ implies that $x \in A$, then $A=X$.
\end{enumerate}
\end{lemma}
\begin{proof}
$(i) \Rightarrow (ii)$: Let $\emptyset \not= A \subseteq X$ and let $x \in A$. Then $\da x \cap A$ is nonempty and so has a minimal element, which is also minimal in $A$. 

$(ii) \Rightarrow (i)$: Let $\emptyset \not= A \subseteq X$ and let $x \in A$, then $x$ dominates a minimal element of $A$.

$(i) \Rightarrow (iii)$: Let $(x_n)_{n\in\N}$ be a decreasing sequence, then $\{x_n \colon n \in \N\}$ has a minimal element, forcing the sequence to be stationary from that point onward.

$(iii) \Leftrightarrow (iv)$: This is clear.

$(iv) \Rightarrow (i)$: We prove the contrapositive. Suppose $X$ is not well-founded, then there exists an $\emptyset \not= A \subseteq X$ without a minimal element. Let $x_1 \in A$, since it is not minimal there exists $x_2 \in A \cap \sda x_1$. Continuing this procedure generates a strictly decreasing sequence $(x_n)_{n\in\N}$.

$(i) \Rightarrow (v)$:  Suppose that $A \subseteq X$ is such that for all $x \in X$, the property that all $y < x$ are in $A$ implies that $x \in A$. Assume that $A \not= X$, then by $(i)$, let $x$ be a minimal element of the nonempty set $A^c$. Then $y \in A$ for all $y < x$ but $x \notin A$ which is a contradiction, hence $A = X$.

$(v) \Rightarrow (iv)$: We prove the contrapositive, so suppose there exists a strictly decreasing sequence $(x_n)_{n\in\N}$ in $X$; let $A = \{x_n \colon n \in \N\}^c$. Let $x \in X$ and assume that $y \in A$ for all $y < x$. If $x \notin A$ then $x = x_n$ for some $n$, but $x_{n+1} < x_n$ and $x_{n+1} \notin A$ contradicting the assumption, so $x \in A$. Hence $A$ satisfies the induction premise but $A \not= X$, therefore $(v)$ fails.
\end{proof}

The next lemma states that the well-founded subsets of a poset $X$ form an ideal in $2^X$, the Boolean algebra of the power set of $X$.

\begin{lemma}\label{L:well-founded closed under unions}
Let $X$ be a poset. Then any subset of a well-founded set in $X$ and any finite union of well-founded sets in $X$ is well-founded. 
\end{lemma}
\begin{proof}
It should be obvious that any subsets of a well-founded set $X$ is well-founded. Let $A, B \subseteq X$ and suppose $A \cup B$ is not well-founded, then it has a strictly decreasing sequence by \Cref{L:well-founded equivalences}. Then either $A$ or $B$ has a strictly decreasing sequence, so either $A$ or $B$ is not well-founded. 
\end{proof}

\subsubsection{Well-founded recursion}

Transfinite recursion similarly generalizes to well-founded recursion, with a proof that is essentially not more complicated than the proof of transfinite recursion. We denote by $D(f)$ the domain of a function $f$.

\begin{theorem}\label{T:well_founded_recursion}
Let $X$ be a well-founded poset, $S$ a set, and let $G$ be a function with codomain $S$ and domain the set of pairs $(x,g)$ with $x \in X$ and $g \colon \sda x \to S$. Then there exists a unique function $F \colon X \to S$ with $F(x) = G(x, F|_{\sda x})$ for all $x \in X$.
\end{theorem}
\begin{proof}
Let $K$ be the set of functions $f$ with codomain $S$ such that $D(f) = \sda x$ for some $x \in X$ and $f(y) = G(y, f|_{\sda y})$ for all $y \in D(f)$. If $f,g \in K$, we claim that $f(x) = g(x)$ for all $x \in D(f) \cap D(g)$. To prove the claim, let $x \in D(f) \cap D(g)$ and suppose $f(y) = g(y)$ for all $y < x$. Then $f(x) = G(x, f|_{\sda x}) = G(x, g|_{\sda x}) = g(x)$. By well-founded induction on $D(f) \cap D(g)$, the claim follows. We can therefore define $F(x)$ to be $f(x)$ if there exists an $f \in K$ with $x \in D(f)$, i.e., $F = \cup K$.

We next show that $F(x) = G(x, F|_{\sda x})$ for all $x \in D(F)$. Clearly $F$ extends all $f \in K$. Let $x \in D(F)$, and let $f \in K$ be such that $x \in D(f)$. Then $F(x) = f(x) = G(x, f|_{\sda x}) = G(x, F|_{\sda x})$. 

To prove that $D(F) = X$, we use well-founded induction. So let $x \in X$ and suppose $y \in D(F)$ for all $y < x$. Define $g(y) := F(y)$ for all $y < x$, so $F$ extends $g$, and extend $g$ to $f$ by defining $f(x) := G(x,g)$. We claim that $f \in K$. Indeed, $f(x) = G(x,g) = G(x, f|_{\sda x})$, and if $y < x$, then 
$$
f(y) = g(y) = F(y) = G(y, F|_{\sda y}) = G(y, g|_{\sda y}) = G(y, f|_{\sda y}).
$$
Hence $x \in D(f) \subseteq D(F)$. By well-founded induction, $D(F) = X$.

For the uniqueness, let $H$ be a function with $D(H) = X$ and $H(x) = G(x, H|_{\sda x})$ for all $x \in X$. Let $A = \{x \in X \colon F(x) = H(x)\}$. We will show that $A = X$ by well-founded induction. Suppose $x \in X$ and $y \in A$ (i.e., $F(y) = H(y)$) for all $y < x$. Then $F(x) = G(x, F|_{\sda x}) = G(x, H|_{\sda x}) = H(x)$, so $x \in A$. By well-founded induction, $A = X$.
\end{proof}

\Cref{T:well_founded_recursion} is typically applied in the following situation: suppose that for each $x \in X$ and $g \colon \sda x \to S$ there is a construction for extending $g$ to a function $\tilde{g} \colon \da x \to S$ (in \Cref{T:well_founded_recursion} this is represented by the function $G$, in the sense that $G(x,g)$ should be thought of as defining the value of $\tilde{g}(x)$). This implies in particular that we can define a function $f$ on all minimal elements of $X$, and then \Cref{T:well_founded_recursion} states that there exists a unique function $F$ extending $f$ to the whole of $X$ that is compatible with the construction $G$.

\section{Artinian and Noetherian vector lattices}

In this section we study Artinian and Noetherian vector lattices. A vector lattice $E$ is said to be \emph{Artinian} if every decreasing sequence of ideals $(I_n)_{n\in\N}$ is stationary, that is, there is an index $k$ such that $I_n = I_k$ for all $n \ge k$. Similarly, a vector lattice $E$ is called \emph{Noetherian} if every increasing sequence of ideals $(I_n)_{n\in\N}$ is stationary. The following proposition is an improvement of \cite[Proposition 5.1]{Kandic-Roelands} stating that an Archimedean vector lattice is finite-dimensional if and only if it is Noetherian.

\begin{proposition}\label{Archimiedean Art=Noeth}
For an Archimedean vector lattice $E$ the following statements are equivalent. 
\begin{enumerate}
    \item $E$ is finite-dimensional.
    \item $E$ is Noetherian.
    \item $E$ is Artinian.
\end{enumerate}
\end{proposition}

\begin{proof}
The equivalence between $(i)$ and $(ii)$ is \cite[Proposition 5.1]{Kandic-Roelands}. If $E$ is finite-dimensional, it is Artinian as it has only finitely many ideals. 

To prove that $(iii)$ implies $(i)$, we argue by contradiction. If $E$ is infinite-dimensional, then by \cite[Theorem 26.10]{Zaanen} there is a sequence of pairwise disjoint nonzero vectors $(x_n)_{n\in \mathbb N}$ in $E$. For each $n\in \mathbb N$ let $I_n$ be the principal ideal generated by $x_n$. For each $m\in \mathbb N$ we define the ideal $J_m:=\sum_{n=m}^\infty I_n$. Since $E$ is Artinian, the decreasing sequence of ideals $(J_m)_{m\in \mathbb N}$ is stationary. Hence, there exists $m_0\in \mathbb N$ such that $J_{m_0}=J_{m_0+1}$ which implies $x_{m_0}\in \sum_{n=m_0+1}^\infty I_n$. However, this is impossible as $x_{m_0}$ is disjoint with $x_n$ for each $n\neq m_0$.  
\end{proof}

By the characterization of Artinian vector lattices and Noetherian vector lattices (see Section 5 and Section 6) it follows that there exist Noetherian vector lattices that are not Artinian and Artinian vector lattices that are not Noetherian. By \Cref{Archimiedean Art=Noeth}, it should be clear that such examples are not Archimedean.

The following result provides a useful characterization of Noetherian vector lattices. 

\begin{proposition}\label{Noetherian characterization}
For a vector lattice $E$ the following statements are equivalent. 
\begin{enumerate}
    \item $E$ is Noetherian. 
    \item Every nonempty family of ideals has a maximal element.     
    \item Every ideal in $E$ is principal. 
\end{enumerate}
\end{proposition}

\begin{proof}
Suppose $(i)$ holds and that there is a nonempty family $\mathcal F$ of ideals of $E$ that does not contain a maximal element. Pick any ideal $I_1\in \mathcal F$. Since $I_1$ is not a maximal element of $\mathcal F$ there exists an ideal $I_2\in \mathcal F$ such that $I_1\subsetneq I_2$. Since $I_2$ is not a maximal element of $\mathcal F$, there exists an ideal $I_3\in \mathcal F$ such that $I_2\subsetneq I_3$. Inductively, we can build an increasing sequence of ideals in $E$, which is not stationary. 

To prove that $(ii)$ implies $(iii)$, let $I$ be an ideal in $E$. Let $\mathcal F$ be the family of all principal ideals contained in $I$, which is nonempty since $\{0\} \in \F$. By $(ii)$, this family has a maximal element $I_0$. If $I_0\neq I$, then there exists a positive nonzero vector $x\in I\setminus I_0$. Now, the ideal $I_0+I_x$ in $I$ is principal and properly contains $I_0$ which is impossible due to the maximality of $I_0$ in $\mathcal F$. Hence, $I_0=I$ is a principal ideal. 

Suppose $(iii)$ holds and consider an increasing sequence of ideals $(I_n)_{n\in\mathbb N}$ in $E$. Then its union $I=\bigcup_{n=1}^\infty I_n$ is a principal ideal in $E$. If $I=I_x$ for some positive vector $x\in E$, then $x\in I_n$ for some $n\in \mathbb N$ implying $I=I_x=I_{n}=I_{n+1}=\cdots$.
\end{proof}

\begin{corollary}\label{Noetherian lattices: maximal ideals}
A nonzero Noetherian vector lattice $E$ has a strong unit and maximal ideals.  
\end{corollary}

\begin{proof}
By \Cref{Noetherian characterization}, $E$ has a strong unit. Therefore, $E$ has maximal ideals by \cite[Theorem~27.4]{Zaanen}.
\end{proof}

An ideal $I$ of a vector lattice $E$ is \emph{minimal} if the zero ideal is the only ideal properly contained in $I$. By \cite[Proposition~II.3.4]{Schaefer}, every minimal ideal is isomorphic to $\mathbb R$. Since every two trivially intersecting ideals are disjoint, it follows that minimal ideals $I$ and $J$ are either equal or disjoint.
The complementary result to \Cref{Noetherian lattices: maximal ideals} in Artinian vector lattices concerns these minimal ideals. This is an immediate consequence of \Cref{L:well-founded equivalences}. 

\begin{proposition}\label{P:artinian_finite_minimal_ideals}
In an Artinian vector lattice $E$ every nonzero ideal contains a minimal ideal. Furthermore, the collection of minimal ideals of $E$ is finite. 
\end{proposition}

\begin{proof}
Suppose $I$ is a nonzero ideal in $E$, and consider the collection of nonzero ideals contained in $I$. This yields a partially ordered set via set inclusion with the property that it does not contain a strictly decreasing sequence, because $E$ is Artinian. By \Cref{L:well-founded equivalences} this set must be well-founded and therefore contains a minimal element. 

Suppose that $E$ has infinitely many minimal ideals and pick a countable subset $\{I_1,I_2,\ldots\}$ of such minimal ideals. For each $k\in \mathbb N$ define the ideal 
$J_m:=\sum_{j=m}^\infty I_j$. Since $E$ is Artinian and the sequence $(J_m)_{m\in \mathbb N}$ is decreasing, there exists $m_0\in \mathbb N$ such that $J_{m_0}=J_{m_0+1}$. This yields that $I_{m_0}\subseteq \sum_{j=m_0+1}^\infty I_j$ and so 
$$I_{m_0} = I_{m_0} \cap\sum_{j=m_0 + 1}^\infty I_j = \sum_{j=m_0+1}^\infty (I_{m_0}\cap I_j).$$ 
Since the intersection of ideals is an ideal, different minimal ideals must be disjoint and we obtain $I_{m_0}=\{0\}$ which is impossible.
\end{proof}

We need the following basic results about subspaces and quotients in vector spaces that will be used later. For the sake of completeness, the proofs are added. 

\begin{lemma}\label{Lemma: equal quotients equal ideals}
Let $E$ be a vector space. Then the following statements hold. 
\begin{enumerate}
    \item For any subset $J$ of $E$ and the quotient projection $\pi\colon E\to E/I$ we have $\pi^{-1}(\pi(J))=I+J$. 
    \item For any subspaces $I$, $J$ and $K$ in $E$ we have $I+J=I+K$ whenever $(I+J)/I=(I+K)/I$.
\end{enumerate}
\end{lemma}

\begin{proof}
    $(i)$: A vector $x$ is an element of $\pi^{-1}(\pi(J))$ if and only if $\pi(x)\in \pi(J)$. Hence, $x\in \pi^{-1}(\pi(J))$ if and only if there exists a $y\in J$ such that $\pi(x)=\pi(y)$ which is equivalent to $x-y\in I$.  

    $(ii)$: Let $\pi \colon E \to E/I$ be the quotient projection. Since $\pi(J)=\pi(I+J)=\pi(I+K)=\pi(K)$, from the assumption it follows that $\pi^{-1}((I+J)/I)=\pi^{-1}((I+K)/I)$. By $(i)$ we have $I+J=I+K$. 
\end{proof}

\begin{proposition}\label{ArtNoeth ideals}
Let $E$ be a vector lattice and $J$ an ideal in $E$. 
Then $E$ is Noetherian (Artinian) if and only if $J$ and $E/J$ are Noetherian (Artinian). \end{proposition}

\begin{proof}
Suppose first that $E$ is Noetherian. Let $(J_n)_{n\in\mathbb N}$ be an increasing sequence of ideals in $J$. Since $J$ is an ideal in $E$ and since $E$ is Noetherian, this sequence needs to be stationary. 

To see that $E/J$ is Noetherian, pick an increasing sequence $(J_n)_{n\in\mathbb N}$ of ideals in $E/J$. If $\pi\colon E\to E/J$ is the canonical lattice homomorphism, then 
$(\pi^{-1}(J_n))_{n\in \mathbb N}$ is an increasing sequence of ideals in $E$. Since $E$ is Noetherian, it needs to be stationary. From surjectivity of $\pi$ it follows that $\pi(\pi^{-1}(J_n))=J_n$ and so $(J_n)_{n\in\mathbb N}$ is stationary. 

Suppose now that $J$ and $E/J$ are Noetherian. Let $(J_n)_{n\in\mathbb N}$ be an increasing sequence of ideals in $E$. Then $(J_k\cap J)_{k\in\mathbb N}$ and $((J+J_k)/J)_{k\in\mathbb N}$ are increasing sequences of ideals in $J$ and $E/J$, respectively. Hence, there exists $n\in\mathbb N$ such that $J_n\cap J=J_{n+1}\cap J=\cdots$ and $(J_n+J)/J=(J_{n+1}+J)/J=\cdots$.  By \Cref{Lemma: equal quotients equal ideals} we conclude that $J_n+J=J_{n+1}+J=\cdots$. 
Since the lattice of ideals of a vector lattice is distributive, for $m\geq n$ we have
\begin{align*}
   J_n&=J_n+(J\cap J_n)=(J_n+J)\cap J_n=(J_m+J)\cap J_n=(J_m\cap J_n)+(J\cap J_n)\\
   &=(J_m\cap J_n)+(J\cap J_m)=J_m\cap (J_n+J)=J_m\cap (J_m+J)=J_m.
\end{align*}
The proof in the Artinian case is the same.
\end{proof}

The following theorem characterizes finite dimensional vector lattices. 

\begin{theorem}\label{T: Noeth and Art iff finite dim}
A vector lattice $E$ is finite-dimensional if and only if it is Artinian and Noetherian. 
\end{theorem}

\begin{proof}
If $E$ is finite-dimensional, it has only finitely many ideals, and hence, it must be Artinian and Noetherian. 
Now, suppose that $E$ is Artinian and Noetherian. If $E$ is infinite-dimensional, then it is possible to construct a strictly decreasing sequence of ideals in $E$. 
Indeed, since $E$ is Noetherian, by \Cref{Noetherian lattices: maximal ideals} it contains a maximal ideal $M_1$ which has codimension one in $E$. 
Since $M_1$ is Noetherian by \Cref{ArtNoeth ideals}, it follows that $M_1$ contains a maximal ideal $M_2$ of codimension one in $M_1$. Recursively, we can now construct a strictly decreasing sequence $(M_n)_{n\in \mathbb N}$ of ideals contradicting the fact that $E$ is Artinian.  
\end{proof}

The following corollary immediately follows from the preceding theorem and \cite[Theorem II.3.6]{Schaefer}. The space $\mathbb R^n_{\lex}$ refers to the lexicographically ordered space described in \cite[Example~II.1.7]{Schaefer}.

\begin{corollary}
A totally ordered nonzero vector lattice $E$ is isomorphic to $\mathbb R^{n}_{\lex}$ for some $n \in \mathbb N$ if and only if it is Artinian and Noetherian. 
\end{corollary}

The intersection of all maximal ideals of a vector lattice $E$ is called the \emph{radical} and it is denoted by $R(E)$. If $E$ does not contain any maximal ideals, then we define $R(E):=E$. 

\begin{proposition}
For a nonzero Noetherian vector lattice or an Artinian vector lattice the following assertions are equivalent. 
\begin{enumerate}
    \item $E$ is Archimedean.
    \item $R(E) = \{0\}$. 
    \item $E$ is isomorphic to $\mathbb R^n$ for some $n\in \mathbb N$ with the coordinatewise order. 
\end{enumerate}
\end{proposition}

\begin{proof} 
Suppose that $E$ is a Noetherian or an Artinian vector lattice. If $(i)$ holds, then $E$ is finite-dimensional by \Cref{Archimiedean Art=Noeth} and hence, $E$ is lattice isomorphic to $\mathbb R^n$ for some $n\in \mathbb N$ with the coordinatewise order. That $(iii)$ implies $(ii)$ is clear and $(ii)$ implies $(i)$ by the Corollary following \cite[Proposition II.3.3]{Schaefer}.   
\end{proof} 

\begin{proposition}\label{finite-dimensional radical}
Let $E$ be a nonzero vector lattice and let $R(E)$ be its radical. Then $R(E)\neq E$ and $E/R(E)$ is finite-dimensional in either of the following cases.  
\begin{enumerate}
    \item $E$ is Noetherian. 
    \item $E$ is Artinian with a strong unit. 
\end{enumerate}
\end{proposition}

\begin{proof}
$(i)$: If $E$ is Noetherian, then by \Cref{ArtNoeth ideals} the vector lattice $E/R(E)$ is also Noetherian. Since it is Archimedean by \cite[Proposition II.3.3]{Schaefer}, by \cite[Proposition 5.1]{Kandic-Roelands}, $E/R(E)$ is finite-dimensional. Since Noetherian vector lattices have strong units by \Cref{Noetherian lattices: maximal ideals},  $E$ contains maximal ideals and so $R(E)\neq E$. 

$(ii)$: Since $E$ has a strong unit, $E$ contains maximal ideals. Suppose that there is a countable set $(M_n)_{n\in\mathbb N}$ of maximal ideals in $E$. For each $n\in \mathbb N$ we define the ideal $I_n:=M_1\cap \cdots \cap M_n$. Since maximal ideals have codimension one in $E$, the codimension of the ideal $I_n$ in $E$ is $n$. The decreasing sequence $(I_n)_{n\in\mathbb N}$ is stationary, so there exists $k\in \mathbb N$ such that $I_k=I_{k+1}$ forcing
$M_1\cap \cdots \cap M_k\subseteq M_{k+1}$. Since $M_{k+1}$ is a prime ideal, we have that $M_k\subseteq M_{k+1}$ or $I_{k-1}\subseteq M_{k+1}$. Since $M_k$ is maximal in $E$, we conclude that $I_{k-1}$ is contained in $M_{k+1}$. Repeating this argument we conclude that $M_1\subseteq M_{k+1}$ which is impossible. This shows that there does not exist an infinite set of maximal ideals from where it follows that the radical has finite codimension in $E$. 
\end{proof}

\begin{theorem}\label{T:finite_sum_local_ideals}
Let $E$ be a nonzero vector lattice with a strong unit such that its radical $R(E)$ has codimension $k\in \mathbb N$. Then $E$ is the sum of $k$ mutually disjoint ideals $I_1,\ldots,I_k$ each of which possesses a unique maximal ideal $M_j$. Furthermore, $I_j$ is of the form $I_j=\mathbb R\circ M_j$. This decomposition of $E$ is unique except for a permutation of indices.
\end{theorem}

\begin{proof}
Since $E/R(E)$ is Archimedean and finite-dimensional, by \Cref{finite-dimensional radical}, $E/R(E)$ is lattice isomorphic to $\mathbb R^k$ ordered componentwise. By \cite[Lemma II.3.8, p.~49]{Schaefer} there exist pairwise disjoint vectors $x_1,\ldots,x_k$ in $E$ such that the image of $x_j$ in $E/R(E)\cong \mathbb R^k$ corresponds to the standard basis vector $e_j$. Let $V$ be the linear span of $\{x_1,\ldots,x_k\}$. Then $E=V+R(E)$. Let $J$ be the ideal generated by $V$. We claim that $J=E$. If $J$ would be a proper ideal in $E$, since $J$ does not contain $R(E)$, then $J$ is certainly not maximal. Hence, since $E$ has a strong unit, $J$ is contained in some maximal ideal $M$. Since $M$ also contains $R(E)$, we conclude $E=J+R(E)=M$ which is impossible. This proves $J=E$.  

Let $J_j$ be the ideal generated by $x_j$. Then $E=J_1\oplus \cdots \oplus J_k$.
Let $M_j:=R(E)\cap J_j$. We claim that the codimension of $M_j$ in $J_j$ is one, showing that $M_j$ is maximal in $J_j$. Indeed, note that 
$$\R^k \cong E/R(E) = (J_1 \oplus \dots \oplus J_k)/R(E) \cong (J_1 + R(E))/R(E) \oplus \dots \oplus (J_k + R(E))/R(E).$$ Since $x_j \notin R(E)$, all the quotient lattices $(J_j + R(E))/R(E)$ have dimension one. The claim now follows from the fact that $(J_j + R(E))/R(E) \cong J_j / (R(E) \cap J_j) = J_j / M_j$. 

If $K$ is any maximal ideal in $J_j$, then $K':=K+\sum_{i\neq j}J_i$ is a maximal ideal in $E$. Hence, $R(E)\subseteq K'$ and so 
$K=K'\cap J_j\supseteq R(E)\cap J_j=M_j$. This implies $K=M_j$. \cite[Proposition II.3.7]{Schaefer} yields that $J_j\cong \mathbb R\circ M_j$. 
The proof of the uniqueness part is the same as in \cite[Theorem II.3.9]{Schaefer}. 
\end{proof}

An application of \Cref{finite-dimensional radical} shows that \Cref{T:finite_sum_local_ideals} is applicable in the case of Noetherian vector lattices or Artinian vector lattices with strong units. 

\section{Lexicographically ordered vector spaces and lattices}

Let $X$ be a poset. By \Cref{L:well-founded closed under unions} the set of functions $f \colon X \to \R$ with well-founded support form a vector space, since $\supp(f+g) \subseteq \supp(f) \cup \supp(g)$. We denote by $m(f) := \min(\supp(f))$ the minimal elements of the support of $f$.
\begin{definition}
Let $X$ be a poset. Then we define the vector space and subset
\[ 
\R^X_{\lex} := \{ f \colon X \to \R \colon \supp(f) \text{ is well-founded} \}, \quad 
\R^X_{\lex+} := \{ f \in \R^X_{\lex} \colon \forall x \in m(f) \ f(x) > 0 \}.
\]
\end{definition}
\begin{lemma}\label{L:well-founded cone}
Let $X$ be a poset. Then 
\begin{align*}
    \R^X_{\lex+} &= \{ f \in \R^X_{\lex} \colon \forall x \in X \ (f(x) < 0 \Rightarrow \exists y \in \sda x \cap m(f) \ f(y) > 0) \} \\
    &= \{ f \in \R^X_{\lex} \colon \forall x \in X \ (f(x) < 0 \Rightarrow \exists y \in \sda x \ f(y) > 0) \}.
\end{align*} Furthermore, $\R^X_{\lex+}$ is a cone.
\end{lemma}
\begin{proof}
Let $f \in \R^X_{\lex +}$ and let $x \in X$ with $f(x) < 0$, then $\da x \cap \supp(f)$ is well-founded so it contains a minimal element $y$ which is also minimal in $\supp(f)$, hence $f(y) > 0$. The second set is clearly a subset of the third set. Finally, if $f \in \R^X_{\lex}$ is such that for all $x \in X$ with $f(x) < 0$, there exists an $y < x$ with $f(y) > 0$, then clearly $f(x) > 0$ for all minimal $x \in \supp(f)$.

It is obvious that $\R^X_{\lex+}$ is invariant under multiplication by positive scalars, and if $0 \not= f \in \R^X_{\lex+}$, then $\supp(f) \not= \emptyset$ so it contains a minimal element, hence $-f \notin \R^X_{\lex+}$, showing that $\R^X_{\lex+} \cap -\R^X_{\lex+} = \{0\}$. It remains to show that $\R^X_{\lex+}$ is closed under addition, so let $f,g \in \R^X_{\lex+}$. Let $x \in m(f+g)$ and suppose that $(f+g)(x) < 0$. Then either $f(x) < 0$ or $g(x) < 0$; say $f(x) < 0$. Hence there is a $y \in \sda x \cap m(f)$ with $f(y) > 0$, and so $g(y) < 0$ since $(f+g)(y) = 0$. Therefore there is a $z < y$ with $g(z) > 0$, and so $(f+g)(z) = g(z) > 0$, so $z \in \supp(f+g)$ contradicting the fact that $x \in m(f+g)$ and $z < x$. Hence $(f+g)(x) > 0$ and so $f+g \in \R^X_{\lex+}$.
\end{proof}

From now on, we will always consider the order generated by $\R^X_{\lex+}$ on $\R^X_{\lex}$. If $A \subseteq X$, $\R^A_{\lex}$ will be considered to be a (partially ordered) subspace of $\R^X_{\lex}$. For $x \in X$, the function $e_x$ is defined by $e_x(y) := \delta_{xy}$.

We will next investigate conditions under which $\R^X_{\lex}$ is a vector lattice.

\begin{definition}
A partially ordered set $X$ is defined to be a \emph{forest} if two incomparable elements have no common upper bound, or equivalently, if $\da x$ is totally ordered for every $x \in X$.
\end{definition}

Let $X$ be a forest and let $f \in \R^X_{\lex}$. We define 
\[
m_+(f) := \{ s \in m(f) \colon f(s) > 0\} \quad \mbox{and} \quad m_-(f) := \{ s \in m(f) \colon f(s) < 0\},
\]and since $m(f)$ is an antichain, $\ua m_+(f)$ and $\ua m_-(f)$ are disjoint, and so $\supp(f) \subseteq \ua m(f) = \ua m_+(f) \sqcup \ua m_-(f)$, therefore\[
f = f \cdot \mathbf{1}_{\supp(f)} = f \cdot \mathbf{1}_{\ua m_+(f)} + f \cdot \mathbf{1}_{\ua m_-(f)},
\]
where the multiplication is defined pointwise. We will now show that the positive part of $f$ equals $f \cdot \mathbf{1}_{\ua m_+(f)}$.

\begin{theorem}\label{T:forest_implies_lattice}
Let $X$ be a forest. Then $\R^X_{\lex}$ is a vector lattice. Furthermore, for $f \in \R^X_{\lex}$, the supremum of $f$ and $-f$ is given by $|f| = f \cdot \mathbf{1}_{\ua m_+(f)} - f \cdot \mathbf{1}_{\ua m_-(f)}$.
\end{theorem}
\begin{proof}
Let $f \in \R^X_{\lex}$ and define $h := f \cdot \mathbf{1}_{\ua m_+(f)}$; to show that $\R^X_{\lex}$ is a vector lattice, it suffices to show that $h = f \vee 0$. We first claim that $h \geq f$ and $h \geq 0$. Indeed, $m(h) = m_+(f)$, and $h(x) = f(x) > 0$ for $x \in m(h)$, showing that $h \geq 0$. Note that $h-f = -f \cdot \mathbf{1}_{\ua m_-(f)}$, and so $m(h-f) = m_-(f)$, and $h(x) = -f(x) > 0$ for $x \in m(h-f)$, hence $h-f \geq 0$. It remains to show that $h$ is the least upper bound of $f$ and $0$.

Let $g \in \R^X_{\lex}$ be such that $g \geq f$ and $g \geq 0$. We have to show that $g \geq h$, which by \Cref{L:well-founded cone} is equivalent to
\[
\mbox{for all }x \in X \mbox{ with }g(x) < h(x) \mbox{, there is an }y < x \mbox{ with }g(y) > h(y).
\]
So let $x \in X$ be such that $g(x) < h(x)$. First assume that $x \notin \ua m(h)$. Then $h$ vanishes on $\da x$ and so $g(x) < h(x) = 0$. Since $g \geq 0$, there is an $y < x$ with $g(y) > 0 = h(y)$.

Finally, if $x \in \ua m(h) = \ua m_+(f)$, then $g(x) < h(x) = f(x)$, and since $g \geq f$, there is an $y < x$ with $g(y) > f(y)$. Since $x \notin \ua m_-(f)$, also $y \notin \ua m_-(f)$, and so $f(y) = h(y)$. Hence $g(y) > f(y) = h(y)$.

The computation
\[
|f| = f \vee -f = (2f \vee 0) - f = 2f \cdot \mathbf{1}_{\ua m_+(f)} - f \cdot \mathbf{1}_{\ua m(f)} = f \cdot \mathbf{1}_{\ua m_+(f)} - f \cdot \mathbf{1}_{\ua m_-(f)}
\]
shows the final statement.
\end{proof}

From now on, $X$ will be a forest. Let $E$ be a subspace of $\R^X_{\lex}$ with the property that if $f \in E$ and $g \in \R^X_{\lex}$ with $\supp(g) = \supp(f)$, then $g \in E$. The above theorem now shows that $E$ is a sublattice of $\R^X_{\lex}$. In particular, we define the sublattice of finitely supported functions on $X$, or equivalently, the subspace spanned by $\{e_x \colon x \in X\}$, by $\Lex(X)$. We will now turn our attention to the more general situation of sublattices $E$ of $\R^X_{\lex}$ containing $\Lex(X)$.

\begin{example}\label{E:Arch}
Let $X$ be a countable antichain. Then $\R^X_{\lex}$ is isomorphic to the space of all real sequences $s$ in the usual Archimedean coordinatewise ordering, $\Lex(X) \cong c_{00}$, and $E$ is isomorphic to a sublattice of $s$ containing $c_{00}$. If $X$ is a finite antichain, then $\R^X_{\lex} = \Lex(X) \cong \R^n$ in the Archimedean coordinatewise ordering, and if $X$ is finite and totally ordered, then $\R^X_{\lex} = \Lex(X) \cong \R^n_{\lex}$.
\end{example}

Let $U \subseteq X$ be an upper set and let $f \in \R^X_{\lex}$ be supported on $U$. By \Cref{T:forest_implies_lattice} $|f|$ is also supported on $U$. Furthermore, suppose that $g \in \R^X_{\lex}$ is such that $0 \leq g \leq f$. If $x \in m(g)$ and $x \notin \ua m(f)$, then $(f-g)(x) = -g(x) < 0$, so by \Cref{L:well-founded cone} there is a $y < x$ with $f(y) = (f-g)(y) > 0$, contradicting that $x \in \ua y \subseteq \ua m(f)$. Hence $m(g) \subseteq \ua m(f)$ and so $g$ is also supported on $U$. We conclude that for each upper set $U$, the set $\R^U_{\lex}$ of functions supported on $U$ is an ideal of $\R^X_{\lex}$.  

Let $E$ be a sublattice of $\R^X_{\lex}$ containing $\Lex(X)$. For each upper set $U$, we define $E(U) := \R^U_{\lex} \cap E$. Since the intersection of a sublattice with an ideal is an ideal in the sublattice, it follows that $E(U)$ is an ideal of $E$. We now investigate whether every ideal in $E$ is of this form, so let $I \subseteq E$ be an ideal. Let $\car(I) \subseteq X$, the \emph{carrier} of $I$, be the union of the supports of $f \in I$. Suppose $x \in \supp(f)$ for some $f \in I$ and $y \geq x$, then $0 \leq e_y \leq e_x \in I_f$ so $e_y \in I_f$. Hence $\car(I)$ is an upper set. The operations $I \mapsto \car(I)$ and $U \mapsto E(U)$ are in general not inverses, because different ideals can have the same carrier. Indeed, consider \Cref{E:Arch}. Then $X$ itself is an upper set and any ideal of $s$ containing $c_{00}$ has $X$ as its carrier. This explains the additional condition in the next theorem. 
\begin{theorem}\label{t:ideals_upper_sets}
Let $X$ be a forest and let $E$ be a sublattice of $\R^X_{\lex}$ containing $\Lex(X)$ such that $m(f)$ is finite for all $f \in E$. The map sending an upper set $U$ of $X$ to the set $E(U)$ of functions in $E$ supported on $U$ is a lattice isomorphism between the upper sets $\U(X)$ of $X$ and the lattice of ideals of $E$, both equipped with the inclusion ordering. Under this isomorphism, principal ideals correspond to upper sets $\ua F$ for some finite antichain $F \subseteq X$, and local ideals correspond to upper sets $\ua x$ for $x \in X$.
\end{theorem}
We need an auxiliary lemma to prove this theorem.
\begin{lemma}\label{L:principal_lex_ideals}
Let $X$ be a forest, and let $f \in \R^X_{\lex +}$ be such that $m(f)$ is finite. Then $I_f = \R^{\ua m(f)}_{\lex}$.
\end{lemma}

\begin{proof}
Note that $I_f = \{g \in \R^X_{\lex} \colon \exists n \in \N \quad |g| \leq nf \} = \{g \in \R^X_{\lex} \colon \exists n \in \N \quad nf \pm g \geq 0\}$. Now suppose $g \in \R^{\ua m(f)}_{\lex}$. Pick $n \in \N$ with $nf(x) > |g(x)|$ (so that $nf(x) \pm g(x) > 0$) for all $x \in m(f)$, then $m(nf \pm g) = m(f)$ and so $nf \pm g > 0$ hence $g \in I_f$.

Conversely, if $g \notin \R^{\uparrow m(f)}_{\lex}$, then there is an $x \in m(g)$ with $x \notin \ua m(f)$, hence it is not possible that there is some $n \in \N$ with $nf \pm g \geq 0$, implying that $g \notin I_f$.
\end{proof}

\begin{proof}[Proof of \Cref{t:ideals_upper_sets}]
Let $U$ be an upper set. Since $e_x \in \Lex(U) \subseteq E$ for all $x \in U$, it follows that $U \subseteq \car(E(U))$. Conversely, if $x \in \car(E(U))$, then there is an $f \in E(U)$ with $f(x) \not= 0$ and so $x \in U$, so $U = \car(E(U))$. 

Let $I \subseteq E$ be an ideal and let $f \in I$. Then $x \in \car(I)$ for all $x \in \supp(f)$ and so $f \in E(\car(I))$, hence $I \subseteq E(\car(I))$. Conversely, let $f \in E(\car(I))$. Then $\supp(f) \subseteq \car(I)$, and for $x \in m(f) \subseteq \supp(f) \subseteq \car(I)$, the function $e_x$ is in the ideal generated by some $g \in I$ with $x \in \supp(g)$ and so $e_x \in I$. Since $m(f)$ is finite, $\mathbf{1}_{m(f)}$ is a linear combination of $\{e_x \colon x \in m(f)\}$ and so $\mathbf{1}_{m(f)} \in I$. For large enough $n$, the function $n\mathbf{1}_{m(f)} - f$ is strictly positive on $m(f)$ and hence a positive element of $E$, showing that $f \in I$. Hence $I = E(\car(I))$, and so the maps $U \mapsto E(U)$ and $I \mapsto \car(I)$ are mutually inverse bijections. They are clearly order preserving and so they are order (and hence lattice) isomorphisms.

To show the correspondence for principal ideals, note that if $f \in E$, then $m(f)$ is a finite antichain and so $I_f = E(\ua m(f))$ by \Cref{L:principal_lex_ideals}, and conversely, if $F$ is a finite antichain of $X$, then $E(\ua F)$ is the principal ideal generated by $\mathbf{1}_F$ (again by \Cref{L:principal_lex_ideals}). Concerning the local ideals, if $x \in X$, then $E(\ua x)$ is local with unique maximal ideal $E(\sua x)$. Conversely, let $L$ be a local ideal. Then it is principal and so corresponds to a finite antichain $F$; it remains to show that $F$ contains exactly one element. If $F = \emptyset$ then $E(F) = \{0\}$ which is not local, and if $F$ contains at least 2 elements $x$ and $y$, then $E(\ua F \setminus \{x\})$ and $E(\ua F \setminus \{y\})$ are different ideals of codimension one, hence maximal, in $E(\ua F)$.
\end{proof}

\section{Representing Artinian vector lattices}

By \Cref{t:ideals_upper_sets}, to show that $\Lex(X)$ is Artinian, it suffices to show that $\U(X)$ is Artinian, i.e., every decreasing sequence in $\U(X)$ is stationary. For that we need some extra conditions on $X$.
\begin{definition}
Let $X$ be a partially ordered set. Then $X$ is defined to have \emph{finite width} if every antichain in $X$ is finite. 
\end{definition}
A (reverse) well-ordered set is a totally ordered set that is (reverse) well-founded. 
\begin{lemma}\label{l:upper_sets_well_ordered}
Let $T$ be a reverse well-ordered set. Then $\U(T)$, the upper sets of $T$ partially ordered by inclusion, is well-ordered.
\end{lemma}
\begin{proof}
It is not difficult to prove that the upper sets of a totally ordered set is totally ordered. Arguing by contradiction, using \Cref{L:well-founded equivalences}$(iv)$, let $(U_n)_{n\in\N}$ be a strictly decreasing sequence of upper sets, hence $U_n \setminus U_{n+1} \not= \emptyset$. Define $t_n := \max (U_n \setminus U_{n+1})$, then $t_n \notin U_{n+1}$ but $t_{n+1} \in U_{n+1}$. If $t_n \geq t_{n+1}$ then since $U_{n+1}$ is an upper set, $t_n \in U_{n+1}$ which is a contradiction. Hence $t_n < t_{n+1}$ showing that $(t_n)_{n \in \N}$ is a strictly increasing sequence in $T$, contradicting the fact that $T$ is reverse well-ordered.
\end{proof}
\begin{theorem}\label{t:lex_artinian_characterisation}
Let $X$ be a forest, so that $\Lex(X)$ is a vector lattice. Then $\Lex(X)$ is Artinian if and only if $X$ is reverse well-founded with finite width.
\end{theorem}
\begin{proof}
Suppose $X$ is reverse well-founded with finite width. By \Cref{t:ideals_upper_sets}, $\Lex(X)$ is Artinian if and only if every decreasing sequence of upper sets is stationary. So let $(U_n)_{n\in\N}$ be a decreasing sequence in $\U(X)$. The set $F$ of maximal elements of $X$ is an antichain, hence finite, and by reverse well-foundedness, every elements of $X$ is below some maximal element $\mu \in F$. For each such $\mu$, consider the sequence $(U_n \cap \da \mu)_{n \in \N}$. Since $\da \mu$ is totally ordered and reverse well-founded, it is reverse well-ordered. \Cref{l:upper_sets_well_ordered} implies that the upper sets of a reverse well-ordered set are well-ordered. Since $(U_n \cap \da \mu)_{n \in \N}$ is a decreasing sequence in the well-ordered set $\U(\da \mu)$, it is stationary, so there is an $n_\mu \in \N$ such that $U_k \cap \da \mu = U_{n_\mu} \cap \da \mu$ for all $k > n_\mu$. Now consider $k > \max\{n_\mu \colon \mu \in F\}$. It follows that
$$
U_k = U_k \cap X = U_k \cap (\bigcup_{\mu \in F} \da \mu) = \bigcup_{\mu \in F} (U_k \cap \da \mu) = \bigcup_{\mu \in F} (U_{n_\mu} \cap \da \mu)
$$
which is independent of $k$, and so $(U_n)_{n\in\N}$ is stationary which shows that $\U(X)$ is Artinian.

Conversely, suppose $X$ is not reverse well-founded. Then there exists a strictly increasing sequence such that $x_1 < x_2 < \ldots$, generating a strictly decreasing sequence $\ua x_1, \ua x_2, \ldots$ of upper sets, showing that $\Lex(X)$ is not Artinian. Finally, suppose that $X$ does not have finite width. Let $\{x_1, x_2, \ldots\}$ be an infinite antichain. Then $U_n := \bigcup_{k = n}^
\infty \ua x_k$ is a strictly decreasing sequence of upper sets, showing that $\Lex(X)$ is not Artinian.
\end{proof}

Let $E$ be an Artinian vector lattice. Our goal is to find a poset $X$ such that $E \cong \Lex(X)$. To find $X$, note that if we already know that $E \cong \Lex(X)$, then \Cref{t:ideals_upper_sets} shows that there is a bijection between $X$ and the set of local ideals of $E$; this bijection is order reversing, since if $x \leq y$, then $L_y \subseteq L_x$ (since smaller elements generate larger upper sets). So given $E$, we let $X$ be the poset of local ideals of $E$ ordered by reverse inclusion. Then $X$ is a forest by \Cref{L:local_ideals_comparable_or_disjoint}. 
\begin{theorem}\label{t:Artinian=Lex(P)}
Let $E$ be an Artinian vector lattice. Then there exists a reverse well-founded forest with finite width such that $E \cong \Lex(X)$.
\end{theorem}
\begin{proof}
For each local ideal $L$ of $E$, we pick a strong unit $e_L$ of $L$. We first claim that for every ideal $I$ of $E$, the set $\{e_L \colon L \subseteq I \mbox{ is a local ideal} \}$ is a Hamel basis of $I$. We will prove the claim by well-founded induction on the set of all ideals of $E$. So let $I \subseteq E$ be an ideal satisfying the following induction hypothesis: 
\[
\text{For each ideal }J\text{ strictly contained in }I\text{, the set }\{e_L \colon L \subseteq J \mbox{ is a local ideal} \}\text{ is a Hamel basis of }J. 
\]
We have to show that $I$ has the same property. So let $0 \not= u \in I$, then $u \in I_u$; since $I_u$ is Artinian with a strong unit, \Cref{T:finite_sum_local_ideals} and \Cref{finite-dimensional radical} show that $I_u = \oplus_{k=1}^n L_k$, and so $u = \sum_{k=1}^n (\lambda_k e_k + u_k)$ with $u_k \in M_k$. Since $M_k$ is strictly contained in $I$, the induction hypothesis implies that $u_k$ is a finite linear combination of $e_L$'s with $L \subseteq M_k \subseteq I$. It follows that $u$ is a finite linear combination of $e_L$'s with $L \subseteq I$, completing the proof of the claim. We now apply the claim to $I = E$ to conclude that $\{e_L \colon L \mbox{ is a local ideal} \}$ is a Hamel basis of $E$.

Let $\L(E)$ be the set of local ideals of $E$ and let $X := \{ \wh{L} \colon L \in \L(E) \}$ be a copy of $\L(E)$ ordered by reverse inclusion: $\wh{L_1} \geq \wh{L_2}$ if and only if $L_1 \subseteq L_2$. Using the above Hamel basis and identifying each $e_L$ with the corresponding function $e_{\wh{L}}$ on $X$ yields a vector space isomorphism of $E$ with $\Lex(X)$. To show that it is a lattice isomorphism, or equivalently an order isomorphism, note that if $u = \sum_{k=1}^n \lambda_k e_{L_k}$, by \Cref{L:local_ideals_comparable_or_disjoint} the positivity of $u$ is purely determined by the positivity of the coefficients $\lambda_k$ of the maximal local ideals. For its corresponding function $f \in \Lex(X)$, since $X$ is ordered by reverse inclusion, this corresponds to the fact that $f$ is positive on the minimal elements of the support of $f$, which is the definition of positivity in $\Lex(X)$.

\Cref{t:lex_artinian_characterisation} shows that $X$ is reverse well-founded with finite width.
\end{proof}

In conclusion, we have the following characterization of Artinian vector lattices.

\begin{theorem}\label{T:characterisation of Artinian vector lattices}
    A vector lattice $E$ is Artinian if and only if it is isomorphic to $\mathrm{Lex}(X)$ for some reverse well-founded forest $X$ with finite width.
\end{theorem}

\section{Representing Noetherian vector lattices}

\begin{theorem}\label{t:char_lex_noetherian}
Let $X$ be a forest and let $E$ be a sublattice of $\R^X_{\lex}$ containing $\Lex(X)$. Then $E$ is Noetherian if and only if $X$ is well-founded with finite width.
\end{theorem}
\begin{proof}
Suppose $X$ is well-founded with finite width. By \Cref{Noetherian characterization} it suffices to show that every ideal is principal. By \Cref{t:ideals_upper_sets}, we have to show that every upper set is of the form $\ua F$ for some finite antichain. So let $U$ be an upper set, and let $F$ be the antichain of minimal elements of $U$. If $U = \emptyset$ then $F = \emptyset$ and $U = \ua F$, and if $U$ is nonempty, then since $X$ is well-founded with finite width, $F$ is finite and $\ua F = U$.

If $X$ is not well-founded, there exists a strictly decreasing sequence such that $x_1 > x_2 > \ldots$, generating the strictly increasing sequence of local ideals $(L_{x_n})_{n\in\N}$ showing that $E$ is not Noetherian. If $X$ does not have finite width, let $\{x_1, x_2, \ldots\}$ be an infinite antichain. Then the sequence of ideals obtained by $I_n := \sum_{k=1}^n L_{x_k}$ is strictly increasing, showing again that $E$ is not Noetherian.
\end{proof}

Let $E$ be a Noetherian vector lattice. As before, we let $\L(E)$ denote the local ideals of $E$, we pick a strong unit $e_L$ for each $L \in \L(E)$, and we consider a copy $X := \{ \wh{L} \colon L \in \L(E) \}$ of the local ideals of $E$ equipped with reverse ordering, so that $X$ is well-founded. To show that $E$ is isomorphic to a sublattice of $\R^X_{\lex}$ containing $\Lex(X)$, it suffices to find an injective lattice homomorphism $T \colon E \to \R^X_{\lex}$ with $Te_L = e_{\wh{L}}$ for every local ideal $L$. For this we require the following definition.

\begin{definition}\label{d:compatible_local_projections}
   Let $E$ be a Noetherian vector lattice. A collection $\{Q_L\}_{\L(E)}$ is called a \emph{compatible set of local projections} if each $Q_L$ is a linear operator from $E$ to $E$ such that 
   \begin{itemize}
       \item $Q_L(E) = L$ and $Q_L^2 = Q_L$;
       \item if $L_1$ and $L_2$ are disjoint, then $Q_{L_1} Q_{L_2} = Q_{L_2} Q_{L_1} = 0$;
       \item if $L_1 \subsetneq L_2$, then $Q_1 Q_2 = Q_2 Q_1 = Q_1$ and $Q_1 e_{L_2} = 0$.
   \end{itemize} 
\end{definition}

The hard part of the characterization of Noetherian vector lattices is showing that a compatible set of local projections exists, which is stated in the following theorem, the proof of which is in \Cref{s:proof_projections}. 
\begin{theorem}\label{t:projections}
Let $E$ be a Noetherian vector lattice. Then a compatible set of local projections exists.
\end{theorem}

For a local ideal $L$, we define $\phi_L \colon L \to \R$ by $\phi_L(\lambda e_L + u) := \lambda$. 

\begin{theorem}\label{t:noetherian_char}
Let $E$ be a Noetherian vector lattice equipped with a compatible set of local projections $\{Q_L\}_{L \in \L(E)}$. Let $X := \{\wh{L} \colon L \in \L(E)\}$ be a copy of the local ideals of $E$ equipped with reverse ordering. Then the map $T \colon E \to \R^X_{\lex}$ defined by $Tu(\wh{L}) := \phi_L(Q_L u)$ is an injective lattice homomorphism mapping $e_L$ to the function $e_{\wh{L}}$. Moreover, under the induced identification of $E$ as a sublattice of $\R^X_{\lex}$, the projection $Q_L$ corresponds to the projection $Q_{\wh{L}}$ defined by $Q_{\wh{L}}(f) := f \cdot \mathbf{1}_{\ua \wh{L}}$. 
\end{theorem}

\begin{proof}
Let $0 \not= u \in E$, then $I_u \not= 0$ and so $I_u = \oplus_{k=1}^n L_k$ by \Cref{T:finite_sum_local_ideals} and \Cref{finite-dimensional radical}, so $u = \sum_{k=1} (\lambda_k e_k + u_k)$ with $\lambda_k \not= 0$ for all $k$. We first investigate $Tu$. Suppose $L$ is a local ideal disjoint with $L_1, \ldots, L_n$ (for $1 \leq k \leq n$, we write $Q_k$ for $Q_{L_k}$). Then $Q_L u = \sum_{k=1}^n Q_L Q_k (\lambda_k e_k + u_k) = 0$, showing that $Tu(\wh{L}) = 0$. If $L$ is not disjoint with $L_1, \ldots, L_n$, it must be comparable with some $L_j$ by \Cref{L:local_ideals_comparable_or_disjoint}. First assume that $L$ strictly contains $L_j$. By the disjointness of $L_j$ and $L_k$, $L$ cannot be contained in some $L_k$, but it could strictly contain another $L_k$. So for each $k$, either $L$ strictly contains $L_k$, or $L$ is disjoint with $L_k$. In the first case $Q_k$ maps into $M_L$ and so $Q_L Q_k = Q_k$ maps into $M_L$ and thus $\phi_L \circ Q_L Q_k = 0$. In the second case $Q_L Q_k = 0$, and so
$$
Tu(\wh{L}) = \phi_L(Q_L u) = \sum_{k=1}^n \phi_L(Q_L Q_k (\lambda_k e_k + u_k)) = 0.
$$
Finally, suppose that $L$ is contained in some $L_j$. Then for $k \not= j$, $L$ is disjoint with $L_k$, therefore $Q_L Q_k = 0$ and so
$$
Tu(\wh{L}) = \phi_L(Q_L u) = \sum_{k=1}^n \phi_L(Q_L Q_k (\lambda_k e_k + u_k)) = \phi_L(Q_L Q_j(\lambda_j e_j + u_j)) = \phi_L(Q_L(\lambda_j e_j + u_j)).
$$
Now suppose $L = L_j$. Then $Tu(\wh{L}) = \lambda_j \not= 0$; this shows in particular that the function $Tu$ is nonzero, so $T$ is injective. If $L \subsetneq L_j$, then $Q_L e_j = 0$ by assumption and so $Tu(\wh{L}) = \phi_L(Q_L(u_j)) = Tu_j(\wh{L})$. This implies in particular that if $u = e_j$, then $Tu(\wh{L}) = 1$ precisely when $L = L_j$ and $Tu(\wh{L}) = 0$ otherwise. 

Using the disjointness of $\{L_1, \ldots, L_n\}$ and the fact that $|\lambda_k e_k + u_k|$ is the supremum of the totally ordered set $\{ \pm (\lambda_k e_k + u_k)\}$, we compute 
$$
|u| = \left| \sum_{k=1}^n \lambda_k e_k + u_k \right| = \sum_{k=1}^n |\lambda_k e_k + u_k| = \sum_{k=1}^n \sgn(\lambda_k) (\lambda_k e_k + u_k).
$$
It follows from the previous deliberations that $T|u|(\wh{L}) = 0$ whenever $L$ is disjoint from $\{L_1, \ldots, L_n\}$ or strictly contains one of them, $T|u|(\wh{L}) = \sgn(\lambda_j) \lambda_j$ whenever $L = L_j$, and $T|u|(\wh{L}) = T \sgn(\lambda_j)u_j (L)$ whenever $L$ is strictly contained in $L_j$. This can be rephrased as follows: if $L \subseteq L_j$, then $T|u|(\wh{L}) = Tu(\wh{L})$ whenever $\lambda_j > 0$ and $T|u|(\wh{L}) = -Tu(\wh{L})$ whenever $\lambda_j < 0$ (and $T|u|(\wh{L}) = 0$ if $L$ is not contained in some $L_j$).

To determine $|Tu|$, we must consider $m(Tu)$, the minimal elements of the support of $Tu$. Since the ordering on $X$ is reversed, by our previous deliberations, this is exactly $\{L_1, \ldots, L_n\}$. By \Cref{T:forest_implies_lattice}, for $L$ contained in $L_j$, if $\lambda_j = Tu(L_j) > 0$, then $|Tu|(\wh{L}) = Tu(\wh{L})$, and if $\lambda_j = Tu(L_j) < 0$, then $|Tu|(\wh{L}) = -Tu(\wh{L})$ (and $|Tu|(\wh{L}) = 0$ if $L$ is not contained in some $L_j$).

The descriptions of $T|u|$ and $|Tu|$ are the same, so $T|u| = |Tu|$ showing that $T$ is a lattice homomorphism.

It remains to prove the final statement, for which we have to show that $TQ_L = Q_{\wh{L}} T$. So let $u \in E$ and let $I$ be a local ideal of $E$. Then $TQ_Lu(\wh{I})  = \phi_I(Q_I Q_L u)$. By \Cref{L:local_ideals_comparable_or_disjoint}, $I$ and $L$ are disjoint, $I \subseteq L$, or $I \supsetneq L$. In the first case $Q_I Q_L = 0$ and so $TQ_Lu(\wh{I})=0$, in the second case $Q_I Q_L = Q_I$ and so $TQ_Lu(\wh{I}) = Tu(\wh{I})$, and in the third case $Q_I Q_L = Q_L$ and $Q_L e_I = 0$, so $TQ_Lu(\wh{I})  = \phi_I(Q_I Q_L u) = 0$. To summarize, if $I \subseteq L$ then $T Q_L u(\wh{I}) = Tu(\wh{I})$ and $T Q_L u(\wh{I}) = 0$ otherwise. In the space of functions on $X$, since the ordering of $X$ is reversed, this corresponds to the function $f(\wh{I}) = Tu(\wh{I})$ if $\wh{I} \geq \wh{L}$ and $f(\wh{I}) = 0$ otherwise, which is exactly $Q_{\wh{L}} Tu$.
\end{proof}

\begin{remark}
In the notation of \Cref{t:noetherian_char}, for $x \in X$ we denote $C_x := I - Q_x$ to be the \emph{cut} determined by $x$. Then \Cref{t:projections} and \Cref{t:noetherian_char} together imply that every Noetherian vector lattice is isomorphic to a sublattice of $\R^X_{\lex}$ containing $\Lex(X)$ invariant under taking cuts, for some reverse well-founded forest $X$. In 1952, Hausner and Wendel showed a very similar result for totally ordered vector lattices (using very different methods): \cite[Theorem~3.1]{HausnerWendel} shows that if $E$ is a totally ordered vector space (which is then clearly a lattice), then there exists a totally ordered set $X$ such that $E$ is isomorphic to a subspace of $\R^X_{\lex}$ (which is then totally ordered and so is a sublattice of $\R^X_{\lex}$) containing $\Lex(X)$ invariant under taking cuts.
\end{remark}

\subsection{Proof of \Cref{t:projections}.}\label{s:proof_projections}

We first require some elementary preliminaries on projections. An idempotent operator $Q\colon V\to V$ on a vector space is a \emph{projection}. It is easy to see that  $I-Q$ is a projection whenever $Q$ is a projection. Moreover, we have $\ker(I-Q)=\ran(Q)$ and $\ran(I-Q)=\ker(Q)$.  Although the following lemma should be known, we add the proof for the reader's convenience.

\begin{lemma}\label{l:projection_equivalences}
Let $V$ be a vector space and $Q$ and $R$ projections on $V$. Then
\begin{enumerate}
    \item $QR = 0 \Leftrightarrow \ran(R) \subseteq \ker(Q)$.   
    \item $QR = Q \Leftrightarrow \ker(R) \subseteq \ker(Q)$.
    \item $RQ = Q \Leftrightarrow \ran(Q) \subseteq \ran(R)$.
\end{enumerate}
\end{lemma}
\begin{proof}
$(i)$: Suppose $QR = 0$ and let $x \in \ran(R)$. Then $Qx = QRx = 0$, so $\ran(R) \subseteq \ker(Q)$. Conversely, if $\ran(R) \subseteq \ker(Q)$, then $QRx =0$ for all $x \in V$ and so $QR = 0$.

$(ii)$:  Since $QR = Q$ can be rewritten as $Q(I-R)=0$, by $(i)$ we have $QR=Q$ if and only if $\ker(R)=\ran(I-R)\subseteq \ker(Q)$. 

$(iii)$: Similarly as in the proof of $(ii)$, since we can rewrite $RQ = Q$ as $(I-R)Q=0$, by $(i)$ we conclude $\ran(Q)\subseteq \ker(I-R)=\ran(R)$. 
\end{proof}
This lemma immediately shows the equivalences in the next definition.
\begin{definition}\label{d:projections_ordering}
Let $V$ be a vector space and let $Q$ and $R$ be projections on $V$, we define $Q \leq R$ to mean that $\ran(Q) \subseteq \ran(R)$ and $\ker(R) \subseteq \ker(Q)$, or equivalently, $QR = RQ = Q$. We define $Q \perp R$ to mean that $\ran(Q) \subseteq \ker(R)$ and $\ran(R) \subseteq \ker(Q)$, or equivalently, $QR = RQ = 0$.
\end{definition}

Given a Noetherian vector lattice $E$, we consider the poset of local ideals $\L(E)$ equipped with the inclusion ordering. On $\L(E)$, define an equivalence relation by $L_1 \sim L_2$ if $\sua L_1 = \sua L_2$ in $\L(E)$. Each equivalence class consists of incomparable local ideals and so they are disjoint by \Cref{L:local_ideals_comparable_or_disjoint} hence the equivalence class is finite (if not, it contains an infinite sequence $L_k$, and then $I_n := \sum_{k=1}^n L_k$ yields a strictly increasing sequence of ideals). For each equivalence class $\{L_1, \ldots, L_n\}$, we consider the ideal $\oplus_{k=1}^n L_k$; we denote the set of such ideals that are order direct sums of all local ideals of an equivalence class by $\overline{\I}(E)$. To construct the projections, we will use reverse well-founded recursion on $\overline{\I}(E) \cup \L(E)$, equipped with inclusion.

The process of reverse well-founded recursion works as follows: given an ideal $I \in \overline{\I}(E) \cup \L(E)$ and projections $Q_J$ on all ideals $J \in \overline{\I}(E) \cup \L(E)$ that strictly contain $I$, we must construct a projection $Q_I$ on $I$. The recursion theorem then asserts that for all ideals $I \in \overline{\I}(E) \cup \L(E)$, there exists a projection $Q_I$ on $I$ which is constructed in the same way from all the projections $Q_J$ on all ideals $J \in \overline{\I}(E) \cup \L(E)$ strictly containing $I$.

\begin{remark}\label{r:construction_from_bigger_ideal}
In some cases we will construct $Q_I$ from a specific $Q_J$ for some $J$ strictly containing $I$. For this it suffices to construct a projection $R$ on $J$ onto $I$. Indeed, in this case $E \cong \ran(Q_J) \oplus \ker(Q_J)$ (as vector spaces) and we will define $Q_I := R \oplus 0$. It is then clear from \Cref{d:projections_ordering} that $Q_I < Q_J$.
\end{remark}

We will first explain how to construct $Q_L$ if $L \in \L(E) \setminus \ol{\I}(E)$. So let $L_1 \in \L(E) \setminus \ol{\I}(E)$ and suppose its equivalence class is $\{L_1, \ldots, L_n\}$, then $n \geq 2$ otherwise $L_1 \in \ol{\I}(E)$. Therefore the ideal $J := \oplus_{k=1}^n L_k \in \ol{\I}(E)$ strictly contains $I$, and so we can use $Q_J$. By \Cref{r:construction_from_bigger_ideal}, it suffices to find a projection $R$ from $J$ onto $L_1$, and we simply define $R$ to be the band projection onto $L_1$. Note that this procedure defines projections $Q_k$ onto $L_k$ for all $1 \leq k \leq n$, and by construction it is clear from \Cref{d:projections_ordering} that these satisfy $Q_k \perp Q_j$ whenever $1 \leq k < j \leq n$.

It remains to consider $I \in \ol{\I}(E)$, and for this we first require a lemma, which roughly states that $\ol{\I}(E)$ ordered by reverse inclusion consists of an initial element, successor elements, and limit elements. 

\begin{lemma}\label{l:initial_successor_limit}
Let $E$ be a nonzero Noetherian vector lattice. Then all elements of $\ol{\I}(E)$ are of one of the following three distinct forms:
\begin{enumerate}
\item $E$.
\item Nonzero maximal ideals $M_L$ of some local ideal $L$.
\item A nonzero intersection of a nonempty chain of local ideals without a minimum.
\end{enumerate}
\end{lemma}
\begin{proof}
Let $I  := \oplus_{i=1}^n L_i \in \ol{\I}(E)$; we define $\mathcal{L}$ to be the set of local ideals strictly containing $I$. By \Cref{L:local_ideals_comparable_or_disjoint}, $\mathcal{L}$ is totally ordered. We claim that in $\mathcal{L}(E)$, for each $1 \leq k \leq n$, $\sua L_k = \L$. To show this, let $J \in \L$. Then $L_k \subseteq I \subsetneq J$, so $J \in \sua L_k$. Conversely, let $J \in \sua L_k$. Since $L_1, \ldots, L_n$ are equivalent, $J \in \sua L_i$ for all $1 \leq i \leq n$, hence $L_i \subseteq M_J$. Then $I = \oplus_{i=1}^n L_i \subseteq M_J$ thus $I \subsetneq J$, hence $J \in \L$. 

We first consider the case that $\mathcal{L}$ has a minimum $L$, so that $\L = \ua L$. Then $I \subsetneq L$ and so $I \subseteq M_L$, hence $M_L$ is nonzero so $M_L = \oplus_{j=1}^m \tilde{L}_j$ by \Cref{T:finite_sum_local_ideals} and \Cref{finite-dimensional radical}. We claim that in $\mathcal{L}(E)$, for each $1 \leq k \leq m$, $\sua \tilde{L}_k = \ua L$. To show this, let $J \in \ua L$. Then $\tilde{L}_k \subseteq M_L \subsetneq L \subseteq J$, so $J \in \sua \tilde{L}_k$. Conversely, let $J \in \sua \tilde{L}_k$. Then $L$ and $J$ both contain $\tilde{L}_k$ so they are comparable by \Cref{L:local_ideals_comparable_or_disjoint}. If $J \subsetneq L$ then $J \subseteq M_L$, so by \Cref{P:local_ideal_in_summand}, $J \subseteq \tilde{L}_j$ for some $1 \leq j \leq m$, which contradicts $\tilde{L}_k \subsetneq J$. Hence $L \subseteq J$, so that $J \in \ua L$. 

We conclude that $\sua L_i = \sua \tilde{L}_j = \ua L$ for each $i$ and $j$, showing that $\tilde{L}_j$ is equivalent to $L_i$. Since $L_1, \ldots, L_n$ is an entire equivalence class, this shows that for each $1 \leq j \leq m$, there exists a $1 \leq i \leq n$ such that $\tilde{L}_j = L_i$. Hence $M_L \subseteq I$, and so $I = M_L$, which shows $(ii)$.

We now consider the case where $\L$ does not have a minimum. Then $I \subsetneq L$ for each $L \in \L$ and so $I \subseteq \cap \L$, hence $\cap \L$ is nonzero so $\cap \L = \oplus_{j=1}^m \tilde{L}_j$ by \Cref{T:finite_sum_local_ideals} and \Cref{finite-dimensional radical}. We claim that in $\mathcal{L}(E)$, for each $1 \leq k \leq m$, $\sua \tilde{L}_k = \L$. To show this, let $J \in \L$. Then $\tilde{L}_k \subseteq \cap \L \subsetneq J$ since $\L$ does not have a minimum, so $J \in \sua \tilde{L}_k$. Conversely, let $J \in \sua \tilde{L}_k$. Let $L \in \L$, then $L$ and $J$ both contain $\tilde{L}_k$ so they are comparable by \Cref{L:local_ideals_comparable_or_disjoint}. Suppose that $J \subseteq L$ for all $L \in \L$. Then $J \subseteq \cap \L$, so by \Cref{P:local_ideal_in_summand}, $J \subseteq \tilde{L}_j$ for some $1 \leq j \leq m$, which contradicts $\tilde{L}_k \subsetneq J$. Hence there exists an $L \in \L$ with $L \subseteq J$, and since $\L$ is an upper set, $J \in \L$.

We conclude that $\sua L_i = \sua \tilde{L}_j = \L$ for each $i$ and $j$, showing that $\tilde{L}_j$ is equivalent to $L_i$. Since $L_1, \ldots, L_n$ is an entire equivalence class, this shows that for each $1 \leq j \leq m$, there exists a $1 \leq i \leq n$ such that $\tilde{L}_j = L_i$. Hence $\cap \L \subseteq I$, and so $I = \cap \L$. If $\L = \emptyset$ then $\cap \L = E$ showing $(i)$, and if $\L \not= \emptyset$ then $I = \cap \L$ is a nonzero intersection of a nonempty chain of local ideals, showing $(iii)$.
\end{proof}

We now define the recursion on $\ol{\I}(E)$ using \Cref{l:initial_successor_limit}. On $E$, we simply take the identity projection. Given a maximal ideal $M_L \in \ol{\I}(E)$ of some local ideal $L$, we have that $M_L \subsetneq L$ and so by \Cref{r:construction_from_bigger_ideal}, it suffices to construct a projection $R \colon L \to M_L$. We define $R(\lambda e_L + u) := u$ and we note for future reference that $e_L \in \ker(Q_{M_L})$. In the final case where $I = \cap \L$ for the chain of local ideals $\L$ strictly containing $I$, our goal is to recursively define $Q_I$ given $\{ Q_L \colon L \in \L\}$. Consider the subspaces $U = \cap_{L \in \L} \ran(Q_L)$ and $V = \cup_{L \in \L} \ker(Q_L)$. Then $U \cap V = \{0\}$. Extend $V$ to a complementary subspace $W$ of $U$, so that $E \cong U \oplus W$ as vector spaces, and define $Q_I(u+w) := u$. By the construction of $Q_I$, we have that $\ker(Q_L) \subseteq \ker(Q_I)$ and $\ran(Q_I) \subseteq \ran(Q_L)$, so that $Q_I < Q_L$.

We have completed the recursive definition of the projections $\{Q_J \colon J \in \overline{\I}(E) \cup \L(E)\}$, and it remains to verify that these projections satisfy the requirements of \Cref{t:projections}. For this, we first fix a local ideal $L$, and we will show that $Q_I < Q_J$ and $e_L \in \ker(Q_I)$ for all ideals $I \in \ol{\I}(E) \cap \L(E) \cap \da M_L$ by reverse well-founded induction. So let $I \in \ol{\I}(E) \cap \L(E) \cap \da M_L$ satisfy the induction hypothesis
\[
\text{For all ideals }J \supsetneq I \text{ in }\ol{\I}(E) \cap \L(E) \cap \da M_L \text{, } Q_J < Q_L \text{ and }e_L \in \ker(Q_J).
\]
If $I = M_L$, then as noted in the construction, $Q_I < Q_L$ and $e_L \in \ker(Q_I)$. In all other cases it is clear from the construction of $Q_I$ that it is constructed as in \Cref{r:construction_from_bigger_ideal} from some ideal $J \supsetneq I$ in $\ol{\I}(E) \cap \L(E) \cap \da M_L$ such that $Q_I < Q_J$. Therefore $Q_I < Q_J < Q_L$, and since $e_L \in \ker(Q_J)$ it follows that $e_L \in \ker(Q_I)$. 

The above argument shows in particular that if $L_1$ and $L_2$ are local ideals with $L_1 \subsetneq L_2$, then $Q_1 Q_2 = Q_2 Q_1 = Q_1$ and $e_2 \in \ker(Q_1)$. It remains to show that if $L_1$ and $L_2$ are disjoint local ideals, then $Q_1 \perp Q_2$. For this, consider the totally ordered and reverse well-founded (hence reverse well-ordered) subsets $\ua L_1$ and $\ua L_2$ in $\L(E)$. The subset of $\ua L_1$ consisting of elements disjoint with $L_2$ is nonempty (it contains $L_1$) and hence has a maximum $L_3$; similarly, let $L_4$ be the maximum of $\ua L_2$ disjoint with $L_1$. Let $L \in \sua L_3$. Then $L_1 \subseteq L$ and by definition of $L_3$, $L$ is not disjoint from $L_2$. By \Cref{L:local_ideals_comparable_or_disjoint}, $L$ is comparable with $L_2$. It cannot be smaller than $L_2$ otherwise $L_1 \subseteq L \subseteq L_2$, so $L_2 \subseteq L$. Conversely, let $L$ be a local ideal containing $L_1$ and $L_2$. Then $L \in \ua L_1$ and so it is comparable with $L_3$. If $L \subseteq L_3$ then $L_3$ would not be disjoint with $L_2$, so $L \in \sua L_3$. Hence $\sua L_3$ consists of all local ideals containing both $L_1$ and $L_2$. By the same argument, $\sua L_4$ is the same set, so $L_3$ and $L_4$ are equivalent: let $\{L_3, L_4, \ldots, L_n\}$ ($n \geq 4$) be its equivalence class. The construction of $Q_3$ and $Q_4$ using the band projections from $I := \oplus_{k=3}^n L_k$ ensures that $Q_3 \perp Q_4$. Since $Q_1 \leq Q_3$ and $Q_2 \leq Q_4$ it follows by \Cref{d:projections_ordering} that $Q_1 \perp Q_2$, as required.

In conclusion, we have the following characterization of Noetherian vector lattices.

\begin{theorem}\label{T:characterisation of Noetherian vector lattices}
    A vector lattice $E$ is Noetherian if and only if it is isomorphic to a sublattice of $\R^X_{\mathrm{lex}}$ containing $\mathrm{Lex}(X)$ for some well-founded forest $X$ with finite width.
\end{theorem}

\section{Archimedean prime Artinian vector lattices}
As mentioned in the preliminary section of this paper, the collection of ideals containing a fixed prime ideal consists entirely of prime ideals and forms a chain. Therefore, it is a natural question to ask in which vector lattices any decreasing sequence of prime ideals is stationary, which we will refer to as prime Artinian vector lattices. The counterpart of this class of vector lattices, referred to as prime Noetherian, where every increasing sequence of prime ideals is stationary, were studied in \cite{Kandic-Roelands}. The main result of this section characterising the uniformly complete prime Artinian vector lattices states that these are precisely the lattices of functions with finite support on some nonempty set.   

\begin{theorem}\label{T:prime Artinian C(K)}
Let $K$ be a compact Hausdorff space. Then $C(K)$ is prime Artinian if and only if $K$ is finite.
\end{theorem}

\begin{proof}
Suppose that $C(K)$ is prime Artinian and let $P$ be an algebraic prime ideal in $C(K)$. Hence, $P$ is also an order prime ideal, see \cite[p. 211]{Zaanen}. Let $f \in C(K)\setminus P$ be such that $0 \le f \le \mathbf{1}$. Denote by $I_n$ for $n \in \mathbb{N}$ the order ideal generated by $f^n$. It follows that $I_{n+1} \subseteq I_n$ for all $n \in \mathbb{N}$ and so $(P+ I_n)_{n\in\N}$ is a decreasing sequence of prime ideals. By assumption, there is an $m \in \mathbb{N}$ so that $P + I_n = P+ I_{n+1}$ for all $n \ge m$. It follows that  
\[
f^m = q + h \le |q| + |h|
\]
for some $q \in P$ and some $h \in I_{m+1}$. Hence, by the Riesz decomposition property there is a positive $p$ in $P$ and a positive $g$ in $I_{m+1}$ such that $f^m = p + g$. Note that $g \le f^m$ and as $g \le \lambda f^{m+1}$ for some $\lambda > 0$, we have $g \le f^m \wedge \lambda f^{m+1} = f^m(\mathbf{1}\wedge \lambda f)$. It follows that 
\[
f^m(\mathbf{1} - \mathbf{1}\wedge \lambda f) = f^m - f^m(\mathbf{1}\wedge \lambda f) \le p,
\]
so $\mathbf{1} - \mathbf{1}\wedge \lambda f \in P$ as $f^m$ is not in $P$ by assumption. Hence, $P + I_1 = C(K)$, and we conclude that $P$ is maximal as for any $f \in C(K) \setminus P$ the order ideal generated by $f$ is also generated by $|f|/\|f\|$.

Let $x \in K$ and consider the ideal $J_x$ consisting of functions $f$ such that $f^{-1}(\{0\})$ is a neighborhood of $x$. Note that if $y \neq x$, then there is an open neighborhood $O_x$ of $x$ whose closure does not contain $y$ as $K$ is a compact Hausdorff space. Hence, by Urysohn's lemma, $M_x$ is the only maximal ideal that contains $J_x$. By \cite[Theorem~2.8]{Gillman-Jerison} we have that $J_x$ is the intersection of all algebraic prime ideals in $C(K)$ that contain $J_x$, so we must have that $J_x = M_x$ as we have shown that all algebraic prime ideals are maximal. By \cite[Theorem~34.1]{Zaanen} it follows that every order prime ideal in $C(K)$ is maximal, and by \cite[Theorem~34.3]{Zaanen} we have that $K$ must be finite.   

Conversely, if $K$ is finite, then $C(K) \cong \mathbb{R}^{\# K}$ has precisely $\# K$ prime ideals, and must therefore be prime Artinian.  
\end{proof}

\begin{theorem}\label{T: prime Artinian unif complete}
Let $E$ be an Archimedean uniformly complete vector lattice. Then $E$ is prime Artinian if and only if it is lattice isomorphic to $c_{00}(\Omega)$ for some nonempty set $\Omega$.
\end{theorem}

\begin{proof}
Suppose that $E$ is a prime Artinian vector lattice and let $x \in E$ be positive. Then the principal order ideal $I_x$ is a uniformly complete vector lattice with a strong unit. Since $I_x$ is prime Artinian by \Cref{ArtNoeth ideals}, it follows from \Cref{T:prime Artinian C(K)} that $I_x$ is finite-dimensional. Hence, $E$ is lattice isomorphic to $c_{00}(\Omega)$ for some nonempty set $\Omega$ by \cite[Theorem~61.4]{Zaanen}. Conversely, if $E$ is lattice isomorphic to $c_{00}(\Omega)$ for some nonempty set $\Omega$, then every prime ideal is maximal by \cite[Proposition~3.1]{Kandic-Roelands}, so that $E$ is prime Artinian.
\end{proof}

\begin{corollary}
Let $X$ be a nonempty topological space. Then the following statements are equivalent.
\begin{itemize}
    \item[$(i)$] $C(X)$ is prime Artinian.
    \item[$(ii)$] $C(X)$ is prime Noetherian.
    \item[$(iii)$] $X$ is finite.
\end{itemize}
\end{corollary}

\begin{proof}
Suppose $C(X)$ is prime Artinian (prime Noetherian). Since $C(X)$ is uniformly complete by \cite[Theorem~43.1]{Zaanen}, it follows from \Cref{T: prime Artinian unif complete} and \cite[Theorem~6.8]{Kandic-Roelands} that $C(X)$ is lattice isomorphic to $c_{00}(\Omega)$ for some nonempty set $\Omega$. As $\mathbf{1}$ is a weak order unit in $C(X)$ and $c_{00}(\Omega)$ has a weak order unit if and only if $\Omega$ is finite, it follows that $X$ must be finite as well. Conversely, if $X$ is finite, then $C(X)$ has exactly $\# X$ prime ideals, so it must be prime Artinian (prime Noetherian). 
\end{proof}

\begin{remark}
If we compare \Cref{T: prime Artinian unif complete} and \cite[Theorem~6.8]{Kandic-Roelands}, we conclude that an Archimedean uniformly complete vector lattice is prime Artinian if and only if it is prime Noetherian. As illustrated by the vector lattice in the next section and the vector lattice of piecewise polynomials $\mathrm{PPol}([a,b])$ in \cite[Section~7]{Kandic-Roelands}, there are examples of non-uniformly complete Archimedean vector lattices with a strong unit that are prime Artinian but not prime Noetherian, and vice versa. 
\end{remark}

\subsection{Vector lattices of root functions and polynomials around zero}
In order to show the difference between prime Artinian and prime Noetherian vector lattices, in particular that these properties do not imply each other in general, a class of vector lattices consisting of continuous piecewise root functions and polynomials are studied. It turns out that we can completely describe the prime ideal structure of these lattices, and use this to indicate the structural differences between these two concepts.

Let $K:=\{\alpha_n \colon n\in\mathbb{N}\}\cup\{0\}$ where $(\alpha_n)_{n\in\N}$ is a strictly decreasing null sequence in $\mathbb{R}$. Further, let $S \subseteq \mathbb{Q}_{>0}$ and put $S_0:= S\cup\{0\}$, with the ordering on $S$ and $S_0$ coming from $\Q$. Define $R_S(K) \subseteq C(K)$ to be the set of continuous functions on $K$ such that for all $f \in R_S(K)$ there is an $\epsilon > 0$ such that 
\[
f(t) = \sum_{k=0}^n a_k t^{s_k} \qquad(t \in [0,\epsilon)\cap K)
\]
for some $n\in \mathbb{N}$, and where $s_k \in S_0$. It follows that the representation of $f$ locally around zero is unique, and that $R_S(K)$ is a vector sublattice of $C(K)$. Indeed, if 
\[
\sum_{k=0}^n a_k t^{s_k} = \sum_{k=0}^m b_k t^{s'_k}\qquad(t\in[0,\epsilon)\cap K)
\]
for some $\epsilon >0$, then there is a number $M \in\mathbb{N}$ large enough so that the substitution $t \mapsto t^M$ yields a polynomial equation. Hence, all the nonzero coefficients must be equal, $n=m$, and $s_k=s'_k$ for all $0\le k\le n$. A similar argument shows that the pointwise supremum of two such functions is of the same type, as such functions can only have finitely many roots. By the lattice version of the Stone-Weierstrass theorem, the vector lattice $R_S(K)$ is uniformly dense in $C(K)$, and the function defined by $f(t):=t\sin(\frac{\pi}{t})$ on $K\setminus\{0\}$ and $f(0):=0$ is continuous on $K$ but is not in $R_S(K)$. Hence, $R_S(K)$ is not uniformly complete. 

By \cite[Lemma~4.1]{Kandic-Roelands} the maximal ideals in $R_S(K)$ are of the form 
\[
M_{t_0}:=\{f\in R_S(K)\colon f(t_0)=0\}=R_S(K)\cap\{f\in C(K)\colon f(t_0)=0\}
\]
for $t_0\in K$. Note that none of the maximal ideals in $R_S(K)$ will be uniformly complete either. Furthermore, by \cite[Lemma~4.1]{Kandic-Roelands} for any prime ideal in in $R_S(K)$ there is a unique $t_0 \in K$ such that $P\subseteq M_{t_0}$. Note that for $t_0 = \alpha_n$ the only prime ideal contained in $M_{\alpha_n}$ is $M_{\alpha_n}$ itself. Indeed, if $P \subseteq M_{\alpha_n}$ is a prime ideal in $R_S(K)$, then for any positive $f \in M_{\alpha_n}$ it follows that $f\wedge \delta_{t_0}=0$ for the point mass function in $t_0$, but $\delta_{t_0}\notin P$, so $f\in P$ and hence, $P=M_{\alpha_n}$. Define 
\[
P_0:=\left\{f\in R_S(K)\colon \mbox{there is an $\epsilon>0$ such that }f(t)=0 \mbox{ for all } t\in [0,\epsilon)\cap K\right\}.
\]
For $P_0$ we can prove the following.
\begin{lemma}\label{L:minimal prime in R_0}
$P_0$ is the unique minimal prime ideal in $R_S(K)$ that is contained in $M_0$.
\end{lemma}

\begin{proof}
It is straightforward to check that $P_0$ is an ideal in $R_S(K)$. Let $f,g \in R_S(K)$ be such that $f \wedge g = 0$. Suppose $f(t)=\sum_{k=0}^n a_kt^{s_k}$ and $g(t)=\sum_{k=0}^m b_kt^{s_k}$ for all $t\in [0,\epsilon)\cap K$, for some $\epsilon>0$. Since $f$ and $g$ can have only finitely many zeros if they are not constantly zero, it follows that there is a $0<\delta<\epsilon$ so that either $f$ or $g$ is zero for $t\in [0,\delta)\cap K$. Hence, either $f \in P_0$ or $g\in P_0$, showing that $P_0$ is a prime ideal in $R_S(K)$. Let $P$ be a prime ideal in $R_S(K)$ such that $P\subseteq M_0$. Suppose that there is a nonnegative function $f$ in $P_0$ that is not in $P$. Then there is an $\epsilon>0$ such that $f(t)=0$ for all $t\in [0,\epsilon)\cap K$.  
Define, for $0<\delta<\epsilon$, the function $g$ by $g(t):=1$ for all $t\in [0,\delta)\cap K$ and $g(t):= 0$, otherwise. It follows that $g \in R_S(K)$ and $f \wedge g = 0$, but $g \notin P$. Hence, $P_0 \subseteq P$ and $P_0$ is therefore the unique minimal prime ideal contained in $M_0$. 
\end{proof}

Define the linear map $\psi\colon M_0 \to \mathrm{Lex}(S)$ by linearizing

\begin{equation}\label{E: psi(S)}
    \psi(t^{s_k}):= e_{s_k}.
\end{equation}
Note that this map is well defined, since every function in $R_S(K)$ is uniquely defined locally around zero.   

\begin{proposition}\label{P: M_0/P = Lex(-N)}\label{P: M_0/P_0}
The linear map $\psi$ given by \eqref{E: psi(S)} is a surjective Riesz homomorphism with kernel $P_0$.
\end{proposition}

\begin{proof}
It should be clear that $\psi$ is surjective. Suppose $f \in M_0$ is such that 
\[
f(t) = \sum_{k=1}^n a_k t^{s_k} = t^{s_1}\left(a_1 + a_2t^{s_2-s_1} + \dots + a_n t^{s_n-s_1}\right), 
\]
where $0<s_1<\dots<s_n$, when $t$ is sufficiently close to $0$ and $a_k \neq 0$ for all $1 \le k \le n$. If $a_1>0$, then for $t$ sufficiently close to $0$ it follows that $f(t)\ge 0$, so $|f|=f$ for $t$ sufficiently close to $0$. Hence, $\psi(f)>\psi(-f)$ and so $\psi(|f|)=\psi(f)=|\psi(f)|$. If $a_1<0$, then similarly, it follows that $|f|=-f$ for $t$ sufficiently close to $0$, so that $\psi(f)<\psi(-f)$ and $\psi(|f|)=\psi(-f)=|\psi(-f)|=|-\psi(f)|=|\psi(f)|$. Hence, the map $\psi$ is a Riesz homomorphism. It is straightforward to check that $P_0 = \ker \psi$. 
\end{proof}

By Theorem~\ref{t:ideals_upper_sets} the ideal structure of $\mathrm{Lex}(S)$ is completely determined by the upper sets $\mathcal{U}_S$ of $S$, that is, a subset $U \in \mathcal{U}_S$ is an upper set if and only if 
\[
I_U := \left\{f\in \mathrm{Lex}(S) \colon \mathrm{supp}(f) \subseteq U \right\}
\]
is an ideal in $\mathrm{Lex(S)}$. The prime ideal structure $\mathcal{P}(R_S(K))$ of $R_S(K)$ is therefore completely described by 
\begin{equation}\label{E: primes in R_S(K)}
\mathcal{P}(R_S(K)) = \mathcal{U}_S \sqcup (K\setminus\{0\}). 
\end{equation}
In order to understand whether the maximal ideal $M_0$ in $R_S(K)$ is prime Noetherian or prime Artinian, we will use that these properties are preserved by taking quotients with (minimal) prime ideals. 

\begin{lemma}\label{L: quotient is Art Noeth}
Let $E$ be a vector lattice. Then the following statements are equivalent.
\begin{itemize}
    \item[$(i)$] $E$ is prime Noetherian (prime Artinian).
    \item[$(ii)$] $E/P$ is Noetherian (Artinian) for all prime ideals $P\subseteq E$.
    \item[$(iii)$] $E/P$ is Noetherian (Artinian) for all mininal prime ideals $P\subseteq E$.
\end{itemize}
\end{lemma}

\begin{proof}
$(i)\Rightarrow (ii)$: Suppose that $E$ is prime Noetherian and let $(Q_n)_{n\in\N}$ be an increasing sequence of ideals in the quotient vector lattice $E/P$ for an arbitrary but fixed prime ideal $P \subseteq E$. For each $k \in \N$ define the ideal $P_k := \pi^{-1}(Q_k)$ for the canonical map $\pi \colon E \to E/P$. Since all these ideals contain the prime ideal $P$, they must be prime ideals themselves as well, and this yields an increasing sequence $(P_n)_{n\in\N}$ of prime ideals by \cite[Theorem~33.3(ii)]{Zaanen}. Since this sequence must be stationary, it follows that $(Q_n)_{n\in\N}$ must be stationary and $E/P$ is Noetherian. It is analogous to show that if $E$ prime Artinian, that then the quotient vector lattice $E/P$ is Artinian for all prime ideals $P \subseteq E$.

$(ii)\Rightarrow (iii)$: This implication should be clear.

$(iii)\Rightarrow (i)$: Suppose that $E/P$ is Noetherian for every minimal prime ideal $P \subseteq E$ and consider an increasing sequence of prime ideals $(P_n)_{n\in\N}$ in $E$. There is a minimal prime ideal $P_0$ such that $P_0 \subseteq P_k$ for all $k \in \N$ by \cite[Theorem~33.7(i)]{Zaanen}. Then $(P_n/P_0)_{n \in \N}$ is an increasing sequence of prime ideals in $E/P_0$, which must be stationary. Hence there is an $m \in \N$ such that $P_{n}/P_0 = P_{n+1}/P_0$ for all $n \ge m$. It follows that for the canonical map $\pi \colon E \to E/P_0$ we have 
\[
P_n = P_n + P_0 = \pi^{-1}(\pi(P_n)) = \pi^{-1}(\pi(P_{n+1})) = P_{n+1} + P_0 = P_{n+1} 
\]
for all $n \ge m$ by \Cref{Lemma: equal quotients equal ideals}, so $E$ must be prime Noetherian. To prove the implication in case $E/P$ is prime Artinian for all minimal prime ideals $P \subseteq E$, we consider a decreasing sequence $(P_n)_{n \in \N}$ of prime ideals in $E/P$. Then the intersection $\bigcap_{n=1}^\infty P_n$ is a prime ideal in $E$ by \cite[Theorem~33.6]{Zaanen}, so it contains a minimal prime ideal $P_0$ by \cite[Theorem~33.7(i)]{Zaanen}. The remainder of the proof is now analogous to the prime Noetherian case.
\end{proof}

Consequently, we obtain the following characterization for vector lattices that are both prime Noetherian and prime Artinian in terms of how large the minimal prime ideals are.

\begin{proposition}\label{P: Art and Noeth iff min primes}
Let $E$ be a vector lattice. Then $E$ is prime Artinian and prime Noetherian if and only if every minimal prime ideal has finite codimension.
\end{proposition}

\begin{proof}
Suppose that $E$ is prime Artinian and prime Noetherian. If $P$ is a minimal prime ideal, then $E/P$ is Noetherian and Artinian by Lemma~\ref{L: quotient is Art Noeth} and hence, must be finite-dimensional by Theorem~\ref{T: Noeth and Art iff finite dim}. Conversely, if $E/P$ is finite-dimensional for all minimal prime ideals $P \subseteq E$, then $E/P$ is both Noetherian and Artinian as it has only finitely many ideals, so $E$ must be prime Noetherian and prime Artinian by Lemma~\ref{L: quotient is Art Noeth}.  
\end{proof}

Combining \Cref{L: quotient is Art Noeth} and \Cref{T:characterisation of Artinian vector lattices} we see that $M_0$ and $R_S(K)$ are prime Artinian if and only if $-S$ is a well-founded forest with finite width. Indeed, any decreasing sequence of prime ideals in $R_S(K)$ that is not contained in $M_0$, must be a singleton set by \eqref{E: primes in R_S(K)}. Similarly, we have that $M_0$ and $R_S(K)$ are prime Noetherian if and only if $S$ is a well-founded forest with finite width by \Cref{L: quotient is Art Noeth} and \Cref{T:characterisation of Noetherian vector lattices}.

Furthermore, in light of Proposition~\ref{P: Art and Noeth iff min primes}, the examples of Archimedean vector lattices that are prime Noetherian but not prime Artinian and vice versa, must contain a minimal prime ideal that has an infinite-dimensional quotient vector lattice. Since this quotient vector lattice is totally ordered and  Noetherian or Artinian, respectively, by Lemma~\ref{L: quotient is Art Noeth}, the lexicographical lattices $\mathrm{Lex}(\N)$ and $\mathrm{Lex}(-\N)$ are of particular interest. These lattices are infinite-dimensional and correspond to the smallest infinite ordinal. Moreover, $\mathrm{Lex}(\N)$ is Noetherian but not Artinian and $\mathrm{Lex}(-\N)$ is Artinian but not Noetherian by \Cref{T:characterisation of Noetherian vector lattices}, since $-\N$ is not well-founded. Therefore, an Archimedean vector lattice $E$ that has a minimal prime ideal $P$ such that either $E/P \cong \mathrm{Lex}(\N)$ or $E/P \cong \mathrm{Lex}(-\N)$ are in this sense the smallest examples. Similarly, the smallest example of an Archimedean vector lattice $E$ that is neither prime Noetherian nor prime Artinian needs to contain a minimal prime ideal such that $E/P \cong \mathrm{Lex}(\mathbb{Z})$, again by Lemma~\ref{L: quotient is Art Noeth}. 

Hence, if we consider $S=\mathbb{N}$, then $M_0/P_0\cong \mathrm{Lex}(\mathbb{N})$, if we wish to construct $\mathrm{Lex}(-\mathbb{N})$ as a quotient space, we take $S$ to be the set obtained via the order isomorphism $f \colon -\mathbb{N} \to S$ given by $f(-n):=\frac{1}{n}$, so that $M_0 / P_0 \cong \mathrm{Lex}(S) \cong \mathrm{Lex}(-\N)$. For the quotient space $\mathrm{Lex}(\mathbb{Z})$, we use the order isomorphism $f \colon \Z \to S$ defined by

\[
f(n) :=  
\begin{cases}  
n+2  & \text{if }  n \ge 0, \\
- \frac{1}{n} & \text{if } n < 0.
\end{cases}
\]
Similarly, note that $M_0 / P_0 \cong \mathrm{Lex}(S) \cong \mathrm{Lex}(\Z)$ in this case. The vector lattice $R_S(K)$ where $S = \mathbb{N}$ contains increasing sequences of principal prime ideals of arbitrary finite length, say, given by 
\[
( t^n ) \subseteq ( t^{n-1} )\subseteq \dots \subseteq ( t ), 
\]
and the vector lattice $R_S(K)$ where $S$ is identified with $-\mathbb{N}$ contains decreasing sequences of prime ideals of arbitrary finite length, say, given by 
\[
( t^{\frac{1}{n}} ) \supseteq ( t^{\frac{1}{n-1}} ) \supseteq \dots \supseteq ( t ).
\]

As a final remark, it is possible to construct examples of these spaces for any countable well-ordered set $T$, using the useful fact that $T$ is order isomorphic to a subset of $\mathbb{Q}_{> 0}$. The idea of the proof is as follows. Suppose that $T=\{t_1,t_2,\ldots\}$ and assign $f_1(t_1):= 1$. Suppose that $\{t_1,\ldots,t_n\}$ has been order embedded into $\Q_{>0}$ and denote that embedding by $f_n$. We will describe a procedure to extend $f_n$ to $\{t_1,\ldots,t_{n+1}\}$. There are three possible cases to consider. If $t_{n+1} < t_k$ for all $1 \le k \le n$, then define $f_{n+1}(t_{n+1}):=\frac{1}{2}\min_{1 \le k \le n}f_n(t_k)$. If $t_{n+1} > t_k$ for all $1 \le k \le n$, then we define $f_{n+1}(t_{n+1}) := 2\max_{1 \le k \le n}f_n(t_k)$. The only other possibility is that both $L_n := \{t_k \colon t_k < t_{n+1}\}$ and $R_n := \{t_k \colon t_{n+1} < t_k\}$ are nonempty. In this case, put $t_i := \max L_n$ and $t_j := \min R_n$ and define $f_{n+1}(t_{n+1}) := \frac{1}{2}(f_n(t_i) + f_n(t_j))$. In this way we can recursively assign the first $n$ terms of $T$ to rational numbers and the function $f \colon T \to \mathbb{Q}_{> 0}$ defined by $f(t_n):= f_n(t_n)$ yields the desired order isomorphism on its image. By identifying $S$ with $T$, via $f$, we obtain the quotient space $M_0 / P_0 \cong \mathrm{Lex}(S) \cong \mathrm{Lex}(T)$.

\subsection*{Acknowledgements}
The first author is supported by the Slovenian Research and Innovation Agency program P1-0222 and grant N1-0217.

\bibliographystyle{alpha}
\bibliography{prime2}

\begin{thebibliography}{HdP80b}

\bibitem[AA02]{AbramovichAliprantis}
Y.~A. Abramovich and C.~D. Aliprantis.
\newblock {\em An invitation to operator theory}, volume~50 of {\em Graduate
  Studies in Mathematics}.
\newblock American Mathematical Society, Providence, RI, 2002.

\bibitem[AB06]{AliprantisBurkinshaw}
C.~D. Aliprantis and O.~Burkinshaw.
\newblock {\em Positive operators}.
\newblock Springer, Dordrecht, 2006.
\newblock Reprint of the 1985 original.

\bibitem[Bil24]{EugeneBilokopytov}
E.~Bilokopitov.
\newblock Characterizations of the projection bands and some order properties
  of the lattices of continuous functions. ar{X}iv preprint.
\newblock {\em https://arxiv.org/abs/2211.11192}, 2024.

\bibitem[Bou19]{Boulabiar19}
K.~Boulabiar.
\newblock An equivalence functor between local vector lattices and vector
  lattices.
\newblock {\em Categ. Gen. Algebr. Struct. Appl.}, 10(1):1--15, 2019.

\bibitem[dP81]{dePagter}
B.~de~Pagter.
\newblock On {$z$}-ideals and {$d$}-ideals in {R}iesz spaces. {III}.
\newblock {\em Nederl. Akad. Wetensch. Indag. Math.}, 43(4):409--422, 1981.

\bibitem[GJ76]{Gillman-Jerison}
L.~Gillman and M.~Jerison.
\newblock {\em Rings of continuous functions}.
\newblock Graduate Texts in Mathematics, No. 43. Springer-Verlag, New
  York-Heidelberg, 1976.
\newblock Reprint of the 1960 edition.

\bibitem[HdP80a]{Huijsmans-dePagter1}
C.~B. Huijsmans and B.~de~Pagter.
\newblock On {$z$}-ideals and {$d$}-ideals in {R}iesz spaces. {I}.
\newblock {\em Nederl. Akad. Wetensch. Indag. Math.}, 42(2):183--195, 1980.

\bibitem[HdP80b]{Huijsmans-dePagter2}
C.~B. Huijsmans and B.~de~Pagter.
\newblock On {$z$}-ideals and {$d$}-ideals in {R}iesz spaces. {II}.
\newblock {\em Nederl. Akad. Wetensch. Indag. Math.}, 42(4):391--408, 1980.

\bibitem[Hui76]{Huijsmans}
C.~B. Huijsmans.
\newblock Riesz spaces for which every ideal is a projection band.
\newblock {\em Indag. Math.}, 38(1):30--35, 1976.
\newblock Nederl. Akad. Wetensch. Proc. Ser. A {\bf 79}.

\bibitem[HW52]{HausnerWendel}
M.~Hausner and J.~G. Wendel.
\newblock Ordered vector spaces.
\newblock {\em Proc. Amer. Math. Soc.}, 3:977--982, 1952.

\bibitem[JK62]{Johnson-Kist}
D.~G. Johnson and J.~E. Kist.
\newblock Prime ideals in vector lattices.
\newblock {\em Canadian J. Math.}, 14:517--528, 1962.

\bibitem[KR22]{Kandic-Roelands}
M.~Kandi\'{c} and M.~Roelands.
\newblock Prime ideals and {N}oetherian properties in vector lattices.
\newblock {\em Positivity}, 26(1):13--26, 2022.

\bibitem[Kun83]{Kunen}
K.~Kunen.
\newblock {\em Set theory}, volume 102 of {\em Studies in Logic and the
  Foundations of Mathematics}.
\newblock North-Holland Publishing Co., Amsterdam, 1983.
\newblock An introduction to independence proofs, Reprint of the 1980 original.

\bibitem[LZ71]{Zaanen}
W.~A.~J. Luxemburg and A.~C. Zaanen.
\newblock {\em Riesz spaces. {V}ol. {I}}.
\newblock North-Holland Publishing Co., Amsterdam-London; American Elsevier
  Publishing Co., New York, 1971.
\newblock North-Holland Mathematical Library.

\bibitem[Sch74]{Schaefer}
H.~H. Schaefer.
\newblock {\em Banach lattices and positive operators}.
\newblock Springer-Verlag, New York-Heidelberg, 1974.
\newblock Die Grundlehren der mathematischen Wissenschaften, Band 215.

\bibitem[Yos42]{Yosida}
K.~Yosida.
\newblock On the representation of the vector lattice.
\newblock {\em Proc. Imp. Acad. Tokyo}, 18:339--342, 1942.

\end{thebibliography}

\end{document}